\documentclass[12pt]{article}
\usepackage{calligra}
\usepackage[T2A]{fontenc}
\usepackage[utf8]{inputenc}
\usepackage[title]{appendix}
\usepackage{blkarray}
\usepackage{amsmath}
\usepackage{amsthm}
\usepackage{amssymb}
\usepackage{enumitem}
\usepackage{mathtools}
\usepackage{graphicx}
\usepackage{comment}
\usepackage[normalem]{ulem}
\usepackage[colorlinks=true,unicode]{hyperref}
\usepackage[pdftex,dvipsnames]{xcolor}
\usepackage{xargs}
\usepackage[colorinlistoftodos,prependcaption,textsize=tiny]{todonotes}

\newcommandx{\unsure}[2][1=]{\todo[linecolor=red,backgroundcolor=red!25,bordercolor=red,#1]{#2}}
\newcommandx{\dialog}[2][1=]{%
  \todo[
    linecolor=red!70!black,
    backgroundcolor=red!8,
    bordercolor=red!70!black,
    #1
  ]{\small #2}}
\newcommandx{\change}[2][1=]{\todo[linecolor=blue,backgroundcolor=blue!25,bordercolor=blue,#1]{#2}}
\newcommandx{\info}[2][1=]{\todo[linecolor=OliveGreen,backgroundcolor=OliveGreen!25,bordercolor=OliveGreen,#1]{#2}}
\newcommandx{\improvement}[2][1=]{\todo[linecolor=Plum,backgroundcolor=Plum!25,bordercolor=Plum,#1]{#2}}
\newcommandx{\hide}[2][1=]{\todo[disable,#1]{#2}}

\numberwithin{equation}{section}

\def\authorfont{\footnotesize}

\def\keywords#1{\par
	\vspace*{8pt}
	{\authorfont{\leftskip24pt\rightskip\leftskip
	\noindent{\it Keywords}\/:\ #1\par}}\par}

\def\ccode#1{\par		
	\vspace*{8pt}
	{\authorfont{\leftskip24pt\rightskip\leftskip
	\noindent #1\par}}\par}

\hypersetup{
  colorlinks=true,
  linkcolor=blue,
  citecolor=blue,
  urlcolor=blue,
  pdftitle={Invariants of the Colored Braid Groupoid}
}

\title{Invariants of the Colored Braid Groupoid}
\author{Illia E. Rohozhkin\thanks{Email: \texttt{rogozhkin.ie@phystech.edu}}}

\theoremstyle{definition}
\newtheorem{theorem}{Theorem}[section]

\newtheorem{definition}{Definition}[section]
\newtheorem{example}{Example}[section]
\newtheorem{remark}{Remark}[section]

\date{\today}

\begin{document}

\maketitle

\begin{abstract}
In this paper, a braid is regarded as a dynamical system of points in the plane. The states of this dynamical system are given by Delaunay triangulations. This construction makes it possible to define an abstract groupoid $\overset{\mathllap{abc}}{\mathcal{G}^{4}_{n+3}}$, which gives a representation of the colored braid groupoid $\text{ColB}(n)$. We define homomorphisms ${f}_{n+3}:\overset{\mathllap{abc}}{\mathcal{G}^{4}_{n+3}} \rightarrow\text{GL}_{2n+1}(\mathbb{Q})$ and ${f}'_{n+3}:\overset{\mathllap{abc}}{\mathcal{G}^{4}_{n+3}} \rightarrow\text{GL}_{2n+1}(\mathbb{C})$, and describe an algorithm for computing the resulting invariants.
\end{abstract}


\ccode{MSC 2020: 57K10, 57K20}

\section{Introduction}

In \cite{Rohozhkin}, an operator invariant of the pure braid group $PB_{n+3}$ was constructed. It will be shown below that this invariant is an invariant of the colored braid groupoid $\text{ColB}(n)$. The group $PB_{n+3}$ is a subgroupoid of the $\text{ColB}(n)$.

We construct several representations of the groupoid $\text{ColB}(n)$ in this paper.

\subsection{Basic definitions}

Consider the lines $\{y=0,z=1\}$ and $\{y=0,z=0\}$ in $\mathbb{R}^3$ and choose $n$ points on each of these lines having abscissas $1,\dots,n$.

\begin{definition}
An $n${\em -strand braid} is a set of $n$ non-intersecting smooth paths connecting the chosen points on the first line with the points of the second line (in an arbitrary order), such that the projection of each of these paths onto the $Oz$-axis represents a diffeomorphism.
\end{definition}

An example of a braid is shown in Fig. \ref{fig:braid}.

\begin{figure}[h]
    \centering
    \includegraphics[width=0.8\textwidth]{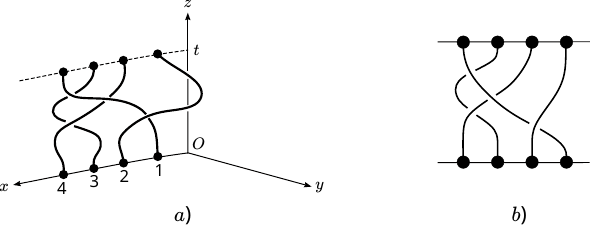}
    \caption{a) A braid; b) its diagram.}
    \label{fig:braid}
\end{figure}

Braids with the same number of strands can be multiplied.

The isotopy classes of braids form the {\em braid group} $Br(n)$. There are also other braid groups. For example the group $B_n$ given by the generators $\sigma_i$ and the Artin relations.

Braids in which each strand connects points with the same abscissas are called {\em pure}. Pure braids form a subgroup $PB_n$ of the braid group $Br(n)$.

More details on classical braids can be found in \cite{Manturov-Knots}.

\subsection{Colored braids}

The strands of a braid can be labeled by different labels. Such braids will be called {\em colored} braids. The product of two such braids is defined only if the labels at the endpoints of the strands of the first braid coincide with the labels at the initial points of the strands of the second braid. Such braids form a groupoid denoted by $\text{ColB}(n)$ \cite{Licata-col-braid}. An example of the product of two colored braids is shown in Fig. \ref{fig:col_braid_mult}.

\begin{figure}[h]
    \centering
    \includegraphics[width=0.8\textwidth]{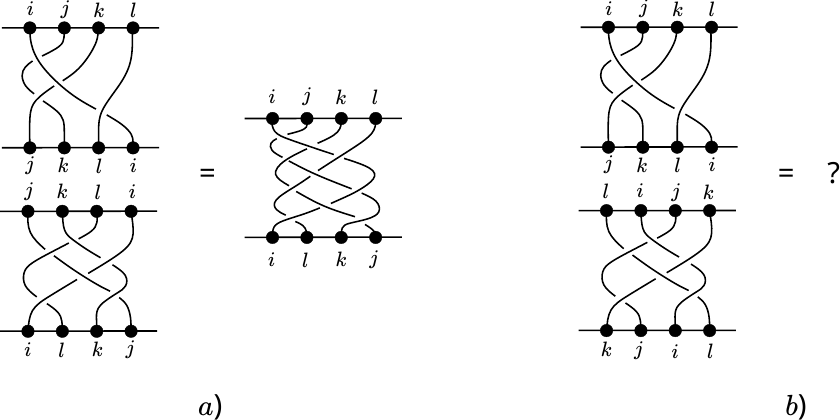}
    \caption{Multiplication of colored braids: a) the product of two braids; b) braids for which the product is not defined.}
    \label{fig:col_braid_mult}
\end{figure}

\subsection{Dynamical systems}

Braids can also be regarded as {\em dynamical systems} of $n$ points. A precise mathematical definition of a dynamical system can be found in \cite{Manturov-Non-Reid} or in the book \cite{Invariants-and-Pictures}. In this paper, by a dynamical system of $n$ points we mean $n$ points moving in the plane over the time interval $0 \le t \le 1$ in such a way that no two points ever pass through one another; that is, the points are not allowed to collide.

A dynamical system can be associated with a braid as follows. Let a braid be given. Consider a plane $P$ parallel to the plane $xOy$ and located at the upper base of the braid. We move this plane downward along the $Oz$-axis and observe the intersection points of the strands of the braid with this plane. The points in the plane $P$ move along the projections of the braid strands onto this plane, thereby forming a dynamical system. This construction is illustrated in Fig. \ref{fig:dynamics}.

\begin{figure}[h]
    \centering
    \includegraphics[width=0.3\textwidth]{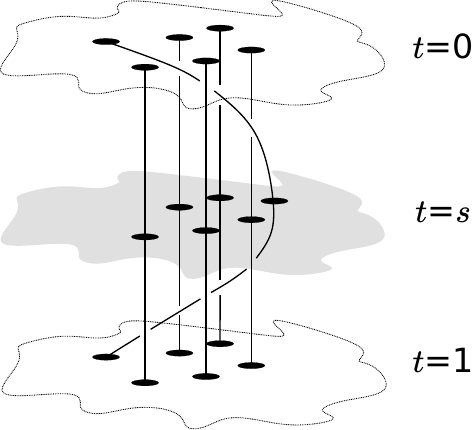}
    \caption{The dynamical system corresponding to a pure braid.}
    \label{fig:dynamics}
\end{figure}

The position of a point in the plane is determined by a pair of its coordinates. A dynamical system can be regarded as an ordered set of {\em states} $D(t)$. Each such state is an ordered collection of pairs of coordinates of all points in the plane $P$ at time $t$ ($0 \le t \le 1$). Here $D(0)$ is the initial state of the dynamical system, and $D(1)$ is its final state. It is important to note that, if the dynamical system corresponds to a braid, then the collections $D(0)$ and $D(1)$ differ from each other only by a permutation of their elements. A pure braid corresponds to the trivial permutation; therefore, one says that the initial and final states of its dynamical system coincide.

With each state of a dynamical system, one can associate a certain geometric construction built on its points. Such constructions may include lines, circles, graphs, and so on. Having chosen a suitable construction, one can distinguish a finite set of time moments $t_i$ ($i=1,\ldots,N$) at which this construction acquires ``unusual'' properties, or {\em singularities}. For example, during the motion of the points, two distinct lines, or circles, constructed from these points may coincide; in a graph, one edge may be replaced by another; and so forth. A braid can then be represented as a sequence of such events. It is important that the number of these events is finite. This makes it possible to construct new invariants for braid groups and braid groupoids.

More details on this approach and on the invariants obtained in this way can be found in the book \cite{Invariants-and-Pictures}, as well as in \cite{Fedoseev-Manturov-Nikonov, Manturov-Non-Reid, Manturov-Nikonov, Rohozhkin}. A similar construction is also used to construct operator invariants of braids and knots; see \cite{Tomotada}.

In this paper, we consider a braid as a dynamical system with three fixed points, and we associate {\em Delaunay triangulations} with the states of this dynamical system.

\subsubsection{Delaunay triangulation}

Recall that any graph can be specified by the set of its vertices $V$ and the set of unordered pairs of these vertices $E$, where each pair is an edge of the graph. All elements of the set $E$ must be distinct and loops are not allowed ($\{x,x\} \not\in E$, $x\in V$).

A graph is called {\em planar} if it can be embedded in the plane in such a way that intersections of edges occur only at its vertices. Such an embedding of a graph in the plane will be called a plane embedding.

A {\em triangulation} is a graph in whose plane embedding every face is bounded by a triangle. Every triangulation is a planar graph. In what follows, to denote a triangulation, we shall use only the set of its edges $E$, assuming that the set of its vertices $V$ can be recovered as the set of elements occurring in the pairs from $E$.

\begin{definition}\label{def:delone}
A {\em Delaunay triangulation} is a planar graph satisfying the following conditions:
\begin{enumerate}
\item \label{delone:3_dots} three points are connected by edges if the interior of the circle passing through these points contains no other points;
\item \label{delone:no_4_dots} there is no quadruple of points on the same circle such that the circle passing through them contains no other points in its interior.
\end{enumerate}
\end{definition}

It follows from the definition of a Delaunay triangulation that its outer face is not always triangular. We shall call a Delaunay triangulation {\em strict} if each of its faces is triangular.

More details on Delaunay triangulations can be found in the book \cite{Aurenhammer-Klein-Lee}.

\begin{example}\label{ex:delone_set}
Examples of strict Delaunay triangulations are shown in Fig. \ref{fig:delaunay_triangles}. The leftmost graph shown in this figure is uniquely determined by the following set of edges:
$$
\{\{1,2\},\{1,3\},\{1,4\},\{1,5\},\{2,3\},\{2,4\},\{2,5\},\{3,4\},\{4,5\}\}.
$$
\end{example}

\subsubsection{Transformations of Delaunay triangulations}

Let a dynamical system corresponding to a braid on $n$ strands be given. If, in the state $D(t)$, the points of the dynamical system satisfy property \ref{delone:no_4_dots} of Definition \ref{def:delone}, then we shall say that the points are in {\em general position}. In what follows, we shall always assume that the points of the dynamical system in the states $D(0)$ and $D(1)$ are in general position.

In the dynamical system under consideration, at a time moment $t_0$, we construct the Delaunay triangulation on its points, taking them as the vertices of the triangulation. Denote this triangulation by $\mathcal{T}_0$. We then start moving the points, simultaneously transforming the triangulation by stretching and contracting the corresponding edges in such a way that their endpoints always coincide with the moving vertices. We shall call such transformations {\em elementary}.

If property \ref{delone:no_4_dots} of Definition \ref{def:delone} fails at a time moment $t_i$, then the graph constructed on these points using only property \ref{delone:3_dots} of Definition \ref{def:delone} is not a Delaunay triangulation. Therefore, in order to work only with Delaunay triangulations, we identify such an event with a new transformation of the triangulation. At this moment, one edge is replaced by another, as shown in Fig. \ref{fig:a_flip}, and as a result we obtain a new Delaunay triangulation $\mathcal{T}_{i}$. Such a transformation is called a {\em flip}. If a flip replaces the diagonal $ik$ of the quadrilateral $(i,j,k,l)$ by the diagonal $jl$, then we shall denote this flip by $ik\rightarrow jl$.

\begin{remark}
In Fig. \ref{fig:a_flip}, at the time moment $t_i$ (in the middle), the edge $\{1,3\}$ is replaced by the edge $\{2,4\}$; we assume that the edge $\{1,3\}$ is already absent at this moment but it is shown by a dashed line in the figure for clarity.
\end{remark}

\begin{figure}[h]
    \centering
    \includegraphics[width=0.7\textwidth]{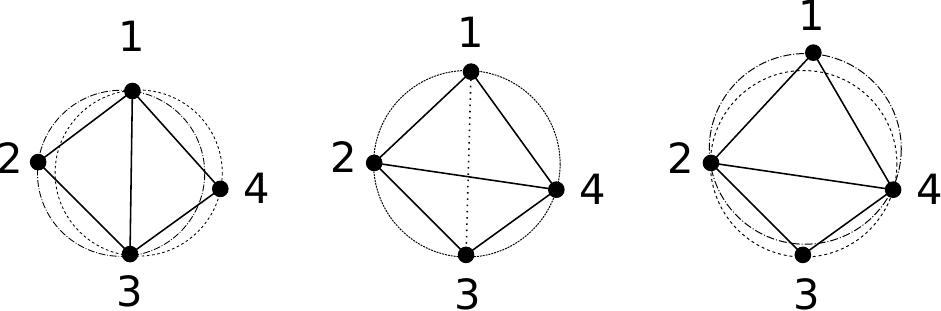}
    \caption{The flip $13 \rightarrow 24$.}
    \label{fig:a_flip}
\end{figure}

\subsubsection{Delaunay triangulations with fixed points}\label{sec:stable_delone}

There is an embedding $Br(n) \subset Br(m)$ for natural numbers $n < m$. A braid from $Br(n)$ can be regarded as a braid in the group $Br(m)$ in which the last $(m-n)$ strands are vertical and separated from the others \cite{Manturov-Knots}. This allows us to add additional {\em fixed} points to any dynamical system corresponding to a braid. These points are needed in order to obtain a more stable construction in which the number of edges of the triangulation always remains constant\footnote{New triangles (and hence new edges) may appear at the boundary of the Delaunay triangulation when one point passes around an extreme point.}. This construction was used in \cite{Rohozhkin}.

The main idea of this construction is to add three additional points to the plane in such a way that all the other points of the dynamical system never leave the exterior triangle formed by these additional points; see Fig. \ref{fig:delaunay_triangles}. Any Delaunay triangulation constructed on the points of such a dynamical system is strict.

The triangulation changes during the motion of the points; however, the number of edges as well as the number of triangles always remains constant in this construction and is determined by the well-known formula for a planar triangulation with a triangular outer face:
\begin{equation}\label{eq:edges_number}
e = 3v - 6,
\end{equation}
where $v$ is the number of vertices of the triangulation, that is, the total number of points of the dynamical system.

\begin{figure}[h]
    \centering
    \includegraphics[width=0.8\textwidth]{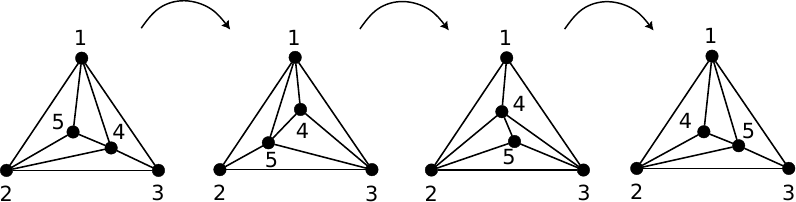}
    \caption{Transformations of Delaunay triangulations constructed on the points of a dynamical system consisting of two moving points $4,5$ and three fixed points $1,2,3$.}
    \label{fig:delaunay_triangles}
\end{figure}

\begin{remark}
In what follows, unless otherwise stated, with each braid on $n$ strands we associate a dynamical system with three additional fixed points. Thus, the total number of points in such a dynamical system is $m=n+3$.
\end{remark}

The three fixed points also allow us to reconstruct uniquely the plane embedding of a Delaunay triangulation, which will be needed below.

\section{The abstract groupoid of colored braids}

We shall consider dynamical systems corresponding to colored braids on $n+3$ strands. Each strand of a colored braid has a unique label $i\in \{1,\ldots,n+3\}$. The points of the dynamical systems have the same labels as the corresponding strands in the braids, and the three additional fixed points have their own labels. We construct Delaunay triangulations on the points of these dynamical systems.

The triangulations undergo elementary transformations and flips during the motion of the points. A flip transforming a Delaunay triangulation $H_i$ (see Example \ref{ex:delone_set}) into a Delaunay triangulation $H_j$ will be denoted by $g^{H_i}_{H_j}$.

In what follows, the total number of possible Delaunay triangulations in a configuration of $n+3$ points will be denoted by $\bar{n}_t$.

Define a multiplication operation on pairs of flips by the rule
\begin{equation}\label{eq:flip_operation}
\left(g^{H_i}_{H_j},g^{H_j}_{H_k}\right) \mapsto g^{H_i}_{H_j}g^{H_j}_{H_k}.
\end{equation}

Clearly, this operation is associative.

\begin{definition}
The groupoid $\overset{\mathllap{abc}}{\mathcal{G}^{4}_{n+3}}$ is the groupoid given by the generators $g^{H_i}_{H_j}$, the associative operation (\ref{eq:flip_operation}), and the following relations:
\begin{enumerate}[label=\Roman*]
\item \ \  $g^{H_i}_{H_i}=1$,\label{rel:identity}
\item \ \ $g^{H_i}_{H_j}g^{H_j}_{H_i}=1$,\label{rel:4_cycle}
\item \ \ $g^{H_i}_{H_j}g^{H_j}_{H_k}g^{H_k}_{H_p}g^{H_p}_{H_q}g^{H_q}_{H_i}=1$,\label{rel:5_cycle}
\item \ \ $g^{H_i}_{H_j}g^{H_j}_{H_k}=g^{H_i}_{H_p}g^{H_p}_{H_k}$,\label{rel:far_commute}
\end{enumerate}

where $H_i,H_j,H_k,H_p,H_q$ are sets of the same cardinality, each of which is a strict Delaunay triangulation,
$$
|H_i \cap H_j| = \begin{cases} 3n+1, \text{ if }i \neq j, \\ 3n+3, \text{ otherwise}, \end{cases}
$$
$a,b,c \in \{1,\ldots,n+3\}$; $|\{a,b,c\}|=3$; $abc$ is the outer face of all Delaunay triangulations; and $i,j,k,p,q \in \{1,\dots,\bar{n}_t\}$.
\end{definition}

\begin{remark}
The superscript $4$ in the notation $\overset{\mathllap{abc}}{\mathcal{G}^{4}_{n+3}}$ reflects the fact that a flip occurs precisely when four points lie on the same circle.
\end{remark}

\subsection{A representation of the colored braid groupoid}

We have shown above that with each colored braid on $n$ strands, and, in general, with each ordinary braid, one can associate a dynamical system of $n+3$ points, and with each such dynamical system one can associate a sequence of flips of Delaunay triangulations. This allows us to associate a sequence of flips with a braid. Hence, we can define a map $f:\text{ColB}(n)\rightarrow \overset{\mathllap{abc}}{\mathcal{G}^{4}_{n+3}}$ as follows:

\begin{equation}\label{eq:group_map}
f(b) = g^{H_1}_{H_2}g^{H_2}_{H_3}\cdots g^{H_{k-1}}_{H_{k}},
\end{equation}
whre $b \in \text{ColB}(n)$; $g^{H_i}_{H_j} \in \overset{\mathllap{abc}}{\mathcal{G}^{4}_{n+3}}$; $i,j,k \in \{1,\dots,\bar{n}_t\}$.

\begin{theorem}\label{theo:col_braids_repr}
The map $f:\text{ColB}(n)\rightarrow \overset{\mathllap{abc}}{\mathcal{G}^4_{n+3}}$, defined by formula (\ref{eq:group_map}), is a homomorphism.
\end{theorem}

\begin{proof}
The compatibility of the map with the multiplication operation is given directly by formula (\ref{eq:group_map}). It remains to prove that braid isotopy is expressed by the relations \ref{rel:identity}--\ref{rel:far_commute}.

Consider two isotopic braids and a plane parallel to $xOy$. We move this plane from the beginning of each braid to its end along the $Oz$-axis, observing the intersection points of this plane with the strands of the braid, as well as the transformations of the Delaunay triangulations on these points. It is not difficult to see that the motions of the points for the first braid may differ from the motions of the points for the second braid by the events shown in Fig. \ref{fig:event_elementary}, Fig. \ref{fig:event_direct_indirect}, Fig. \ref{fig:event_pentagon}, and Fig. \ref{fig:event_commutativity}. These events correspond precisely to the relations \ref{rel:identity}--\ref{rel:far_commute}.

\begin{figure}[!htbp]
    \centering
    \includegraphics[width=0.9\textwidth]{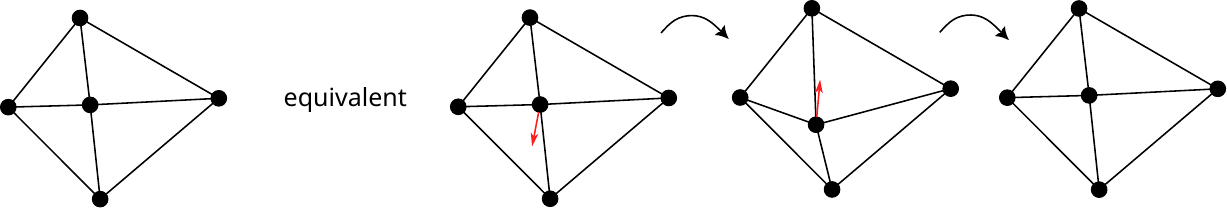}
    \caption{Two isotopic braids may differ from each other by elementary transformations; the arrows indicate the directions of motion of the point.}
    \label{fig:event_elementary}
\end{figure}

\begin{figure}[!htbp]
    \centering
    \includegraphics[width=0.9\textwidth]{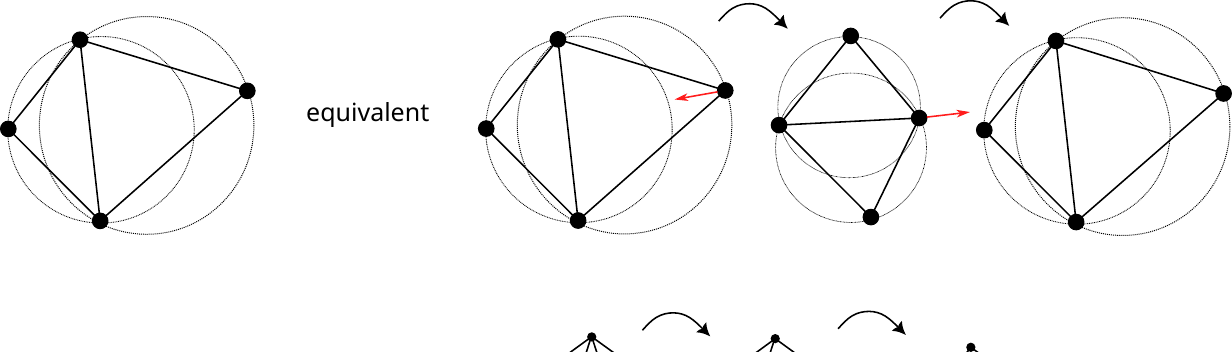}
    \caption{Two isotopic braids may differ from each other by motions of points in which two consecutive flips occur in the same quadrilateral whose interior contains no other points; the arrows indicate the directions of motion of the points; the circles are shown for clarity.}
    \label{fig:event_direct_indirect}
\end{figure}

\begin{figure}[!htbp]
    \centering
    \includegraphics[width=0.85\textwidth]{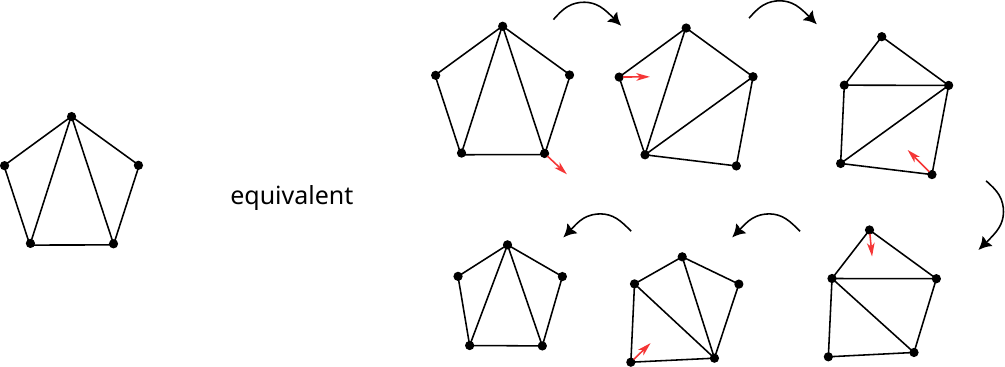}
    \caption{Two isotopic braids may differ from each other by motions of points that lead to a cyclic change of the triangulation of a pentagon whose interior contains no other points; the arrows indicate the directions of motion of the points.}
    \label{fig:event_pentagon}
\end{figure}

\begin{figure}[!htbp]
    \centering
    \includegraphics[width=0.75\textwidth]{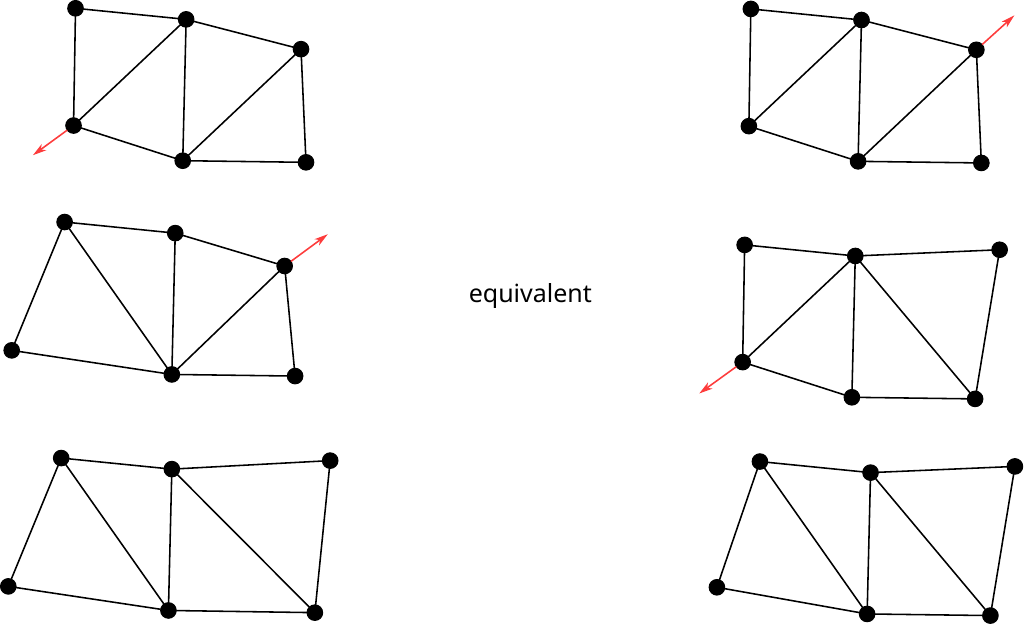}
    \caption{Two isotopic braids may differ from each other by the order of two flips that occur in different quadrilaterals whose interiors contain no other points and which intersect in at most two points; the arrows indicate the directions of motion of the points.}
    \label{fig:event_commutativity}
\end{figure}

It turns out that the events listed above are sufficient to express all other events that may arise during the motion of points in the plane and may lead to other relations on braids. We do not give the proof of this part in the present paper, but a more detailed discussion of its idea can be found in the book \cite{Invariants-and-Pictures}.
\end{proof}

\subsection{A matrix representation of the groupoid $\overset{\mathllap{abc}}{\mathcal{G}^{4}_{n+3}}$}\label{sec:groupoid_repr}

The construction described in Section \ref{sec:stable_delone} was already used in \cite{Rohozhkin}, where a matrix representation of the pure braid group by matrices of size $(2n+1)\times(2n+1)$ was obtained. We shall construct a representation of the groupoid $\overset{\mathllap{abc}}{\mathcal{G}^{4}_{n+3}}$ by the same matrices.

The upper index $H_i$ and the lower index $H_j$ of each generator $g^{H_i}_{H_j} \in \overset{\mathllap{abc}}{\mathcal{G}^{4}_{n+3}}$ are Delaunay triangulations, which are specified by sets of edges. From each such set of edges, one can form the set of triangles of the triangulation. Since we consider only strict Delaunay triangulations with three fixed points, the outer triangular face can be discarded. Assign to each point with index $i$ a variable $\zeta_i$. Variables with distinct indices must be pairwise distinct and their values must be chosen from a field containing at least $n+3$ pairwise distinct elements. We take the field $\mathbb{Q}$ as such a field (although other fields may also be used). If a Delaunay triangulation is specified by a set of edges $H$, then the corresponding set of triangles will be denoted by $\mathbf{T}$. Elements of the set $\mathbf{T}$ will be denoted by the symbol $\Delta$. These notations will also be used with indices, whose meaning will be clear from the context; for example, $H_i$, $\mathbf{T}_i$, and $\Delta_{pqr}$.

It follows from formula (\ref{eq:edges_number}) and Euler's formula for a connected planar graph (i.e. $v-e+f=2$, where $v$ is the number of vertices, $e$ is the number of edges and $f$ is the number of faces of the graph) that the number of elements in the set $\mathbf{T}$ excluding the outer face is given by
$$
|\mathbf{T}| = 2n+1.
$$

Each element $\Delta \in \mathbf{T}$ can be regarded as a vector and the whole set $\mathbf{T}$ as a basis of the vector space $\mathbb{Q}^{(\mathbf{T})}$ over the field $\mathbb{Q}$. To each generator $g^{H_t}_{H_p} \in \overset{\mathllap{abc}}{\mathcal{G}^{4}_{n+3}}$, we assign a linear operator $\gamma_{\mathbb{Q}^{(\mathbf{T}_t)},\mathbb{Q}^{(\mathbf{T}_p)}}:\mathbb{Q}^{(\mathbf{T}_t)} \rightarrow \mathbb{Q}^{(\mathbf{T}_p)}$, which is defined by the following formulas:

\begin{gather}
\gamma(\Delta)=\Delta \mbox{ for any } \Delta\in \mathbf{T}_t\cap \mathbf{T}_p, \label{eq:operator-1}\\
\gamma(\Delta_{ijk})=\frac{\zeta_{i} - \zeta_{l}}{\zeta_{i} - \zeta_{k}}\Delta_{ijl} + \frac{\zeta_{l} - \zeta_{k}}{\zeta_{i} - \zeta_{k}}\Delta_{jkl}, \label{eq:operator-2}\\
\gamma(\Delta_{ikl})=\frac{\zeta_{i} - \zeta_{j}}{\zeta_{i} - \zeta_{k}}\Delta_{ijl} + \frac{\zeta_{j} - \zeta_{k}}{\zeta_{i} - \zeta_{k}}\Delta_{jkl}, \label{eq:operator-3}
\end{gather}

\begin{remark}
The indices of each element $\Delta$ are arranged in lexicographic order. This makes it possible to order the elements of the bases lexicographically. For example, the element $\Delta_{ijk}$ ($i<j<k$) is placed before the element $\Delta_{i'j'k'}$ in the basis if $i < i'$; or if $i = i'$ and $j < j'$; or if $i = i'$, $j = j'$ and $k < k'$.
\end{remark}

The ordered basis of the vector space $\mathbb{Q}^{(\mathbf{T}_t)}$ will be denoted by $\mathbf{f}_t$. The matrix of the linear operator $\gamma_{\mathbb{Q}^{(\mathbf{T}_t)},\mathbb{Q}^{(\mathbf{T}_p)}}$ will be denoted by $A_{\mathbf{f}_t,\mathbf{f}_p}$. A flip inside a quadrilateral replaces two triangles by two others. Therefore, the bases $\mathbf{f}_t,\mathbf{f}_p$, if they are obtained from a single generator of the groupoid $\overset{\mathllap{abc}}{\mathcal{G}^{4}_{n+3}}$, differ from each other by four elements, from which one can uniquely recover the quadrilateral $\{i,j,k,l\}$ in which the flip occurs. If the remaining elements of the bases are clear from the context and the corresponding flip changes the diagonal $ik$ to the diagonal $jl$, then we shall denote the matrix $A_{\mathbf{f}_t,\mathbf{f}_p}$ more simply by $A_{ikjl}$. Such a matrix is defined by the following formula:

\begin{equation}
  A_{ikjl} = 
    \begin{pmatrix}
      * & ... & * & ... & * & ... & * \\
      \  & \  & \  & \vdots & \  & \  & \   \\
      * & ... & \frac{\zeta_i-\zeta_l}{\zeta_i-\zeta_k} & ... & \frac{\zeta_i-\zeta_j}{\zeta_i-\zeta_k} & ... & * \\
      \  & \  & \  & \vdots & \  & \  & \   \\
      * & ... & \frac{\zeta_l-\zeta_k}{\zeta_i-\zeta_k} & ... & \frac{\zeta_j-\zeta_k}{\zeta_i-\zeta_k} & ... & * \\
      \  & \  & \  & \vdots & \  & \  & \  \\
      * & ... & * & ... & * & ... & * \\
    \end{pmatrix},
  \label{equation:matrix_n_2}
\end{equation}
where $i,j,k,l \in \{1,2,\ldots,n+3\}$ and the {\em asterisk} is equal either to $0$ or to $1$ depending on the order of the vectors in the bases.

\begin{remark}\label{rem:dots_algo}
To construct this matrix, one must be able to determine uniquely the indices $i,j,k,l$ of the points in an arbitrary quadrilateral. The following algorithm is used for this purpose. The point $i$ is taken to be the point with the smallest index. Then the point $k$ is determined automatically as the second endpoint of the diagonal $ik$. As the point $j$, one may choose the point with the smaller index among the two remaining points. The point $l$ is then the remaining point. If the flip occurring in the quadrilateral is $ik\rightarrow jl$, then we simply substitute the indices obtained in this way and the values of the corresponding variables into formula (\ref{equation:matrix_n_2}). If the flip is $jl \rightarrow ik$, then in formula (\ref{equation:matrix_n_2}) one must replace the index $i$ by the index $j$ and the index $k$ by the index $l$, while also interchanging the corresponding variables. For example, if the vertices of the quadrilateral are arranged in the plane in the order $(5,3,1,8)$, then the indices are as follows: $i=1$, $j=3$, $k=5$, and $l=8$. In this case, the choice of indices does not depend on whether the vertices of the quadrilateral are arranged in the plane clockwise or counterclockwise.
\end{remark}

We construct a map $f_{n+3}:\overset{\mathllap{abc}}{\mathcal{G}^{4}_{n+3}} \rightarrow \text{GL}_{2n+1}(\mathbb{Q})$. For each generator $g^{H_i}_{H_j}\in \overset{\mathllap{abc}}{\mathcal{G}^{4}_{n+3}}$, one can construct a linear map by formulas (\ref{eq:operator-1})--(\ref{eq:operator-2}) denoted by $\gamma^{\mathbf{f}_i}_{\mathbf{f}_j}$, which maps the vector space with basis $\mathbf{f}_i$ to the vector space with basis $\mathbf{f}_j$ and is determined by its matrix $A^{H_i}_{H_j} \in \text{GL}_{2n+1}(\mathbb{Q})$ according to formula (\ref{equation:matrix_n_2}). Thus, to each word in $\overset{\mathllap{abc}}{\mathcal{G}^{4}_{n+3}}$, one can assign a composition of such linear maps according to the rule:
$$
g^{H_p}_{H_q}g^{H_q}_{H_r}\cdots g^{H_t}_{H_u}g^{H_u}_{H_v} \mapsto \gamma^{\mathbf{f}_u}_{\mathbf{f}_v}\gamma^{\mathbf{f}_t}_{\mathbf{f}_u}\cdots\gamma^{\mathbf{f}_q}_{\mathbf{f}_r}\gamma^{\mathbf{f}_p}_{\mathbf{f}_q}(\mathbf{f}_p).
$$

This composition can be rewritten in matrix form as follows:
$$
\gamma^{\mathbf{f}_u}_{\mathbf{f}_v}\gamma^{\mathbf{f}_t}_{\mathbf{f}_u}\cdots\gamma^{\mathbf{f}_q}_{\mathbf{f}_r}\gamma^{\mathbf{f}_p}_{\mathbf{f}_q}(\mathbf{f}_p) = \mathbf{f}_u A^{H_u}_{H_v}A^{H_t}_{H_u}\cdots A^{H_q}_{H_r}A^{H_p}_{H_q}.
$$

We define the map $f_{n+3}$ so that it assigns to each word in $\overset{\mathllap{abc}}{\mathcal{G}^{4}_{n+3}}$ the product of matrices given by the following formula:
\begin{equation}\label{eq:h}
f_{n+3}(g^{H_p}_{H_q}g^{H_q}_{H_r}\cdots g^{H_t}_{H_u}g^{H_u}_{H_v}) = A^{H_u}_{H_v}A^{H_t}_{H_u}\cdots A^{H_q}_{H_r}A^{H_p}_{H_q}.
\end{equation}

\begin{theorem}
The map $f_{n+3}$ is a homomorphism.
\end{theorem}

\begin{proof}
To prove the theorem, it is necessary to define the map $f_{n+3}$ on the generators of the groupoid $\overset{\mathllap{abc}}{\mathcal{G}^{4}_{n+3}}$ and to show that the relations of this groupoid hold for the corresponding matrices.

The map is defined on the generators of the groupoid by the algorithm described in Remark \ref{rem:dots_algo}. The identity element is mapped to the identity matrix $I_{2n+1}$. Relation \ref{rel:identity} is trivial. The validity of relations \ref{rel:4_cycle}--\ref{rel:far_commute} for these matrices was proved in \cite{Rohozhkin}, in the proof of Lemma 4.1.
\end{proof}

\subsection{A representation of $\overset{\mathllap{abc}}{\mathcal{G}^{4}_{n+3}}$ by orthogonal matrices}

In \cite{Korepanov-pentagon}, an example of the pentagon equation in orthogonal matrices is given, which has the following form:
\begin{gather}
\left(\begin{array}{rrr}
1 & 0 & 0 \\
0 & \cos{\phi_{1345}} & \sin{\phi_{1345}} \\
0 & -\sin{\phi_{1345}} & \cos{\phi_{1345}}
\end{array}\right)
\left(\begin{array}{rrr}
\cos{\phi_{1234}} & -\sin{\phi_{1234}} & 0 \\
\sin{\phi_{1234}} & \cos{\phi_{1234}} & 0 \\
0 & 0 & 1
\end{array}\right) \notag \\
= \left(\begin{array}{rrr}\label{eq:pentagon_korep}
\cos{\phi_{1235}} & -\sin{\phi_{1235}} & 0 \\
\sin{\phi_{1235}} & \cos{\phi_{1235}} & 0 \\
0 & 0 & 1
\end{array}\right)
\left(\begin{array}{rrr}
1 & 0 & 0 \\
0 & \cos{\phi_{2345}} & \sin{\phi_{2345}} \\
0 & -\sin{\phi_{2345}} & \cos{\phi_{2345}}
\end{array}\right) \\
\cdot \left(\begin{array}{rrr}
\cos{\phi_{2345}} & 0 & \sin{\phi_{2345}} \\
0 & 1 & 0 \\
-\sin{\phi_{2345}} & 0 & \cos{\phi_{2345}}
\end{array}\right),\notag 
\end{gather}
where $\cos{\phi_{ijkl}}=\sqrt{\frac{\zeta_{il}\zeta_{jk}}{\zeta_{ik}\zeta_{jl}}}$ and $\sin{\phi_{ijkl}}=\sqrt{\frac{\zeta_{ij}\zeta_{kl}}{\zeta_{ik}\zeta_{jl}}}$ ($\zeta_{ij} = \zeta_i - \zeta_j$, $i,j \in \{1,\dots,5\}$), with all square roots taken to be positive and the following constraint must also be satisfied: $\zeta_1 > \zeta_2 > \zeta_3 > \zeta_4 > \zeta_5$.

By direct enumeration and computation, one can verify that equation (\ref{eq:pentagon_korep}) holds if the variables are ordered not only in decreasing order but also in other orders. However, it does not hold for every order.

We adapt these matrices in order to obtain a representation of the groupoid $\overset{\mathllap{abc}}{\mathcal{G}^{4}_{n+3}}$. Since, in the dynamical systems under consideration, the points corresponding to the strands of a braid may change their relative positions arbitrarily, we need equation (\ref{eq:pentagon_korep}) to hold for any permutation of the points.

In what follows, we repeat the analogous procedure described in the previous subsection. To each point we assign a variable $\zeta_i \in \mathbb{C}$, $i \in \{1,\dots,n+3\}$. All these variables must be pairwise distinct so that no division by zero occurs. We shall use the same notation for triangles, sets of triangles and basis vectors of vector spaces. To each triangle we assign a vector and to the whole triangulation we assign the vector space $\mathbb{C}^{\mathbf{T}}$ over the field $\mathbb{C}$.

We now extend the operators from \cite{Korepanov-pentagon}, whose matrices give equation (\ref{eq:pentagon_korep}). Without loss of generality, up to rotations, consider all possible quadrilaterals in the plane in which flips may occur; see Fig. \ref{fig:quadrilaterials}.
\begin{figure}[h]
    \centering
    \includegraphics[width=0.6\textwidth]{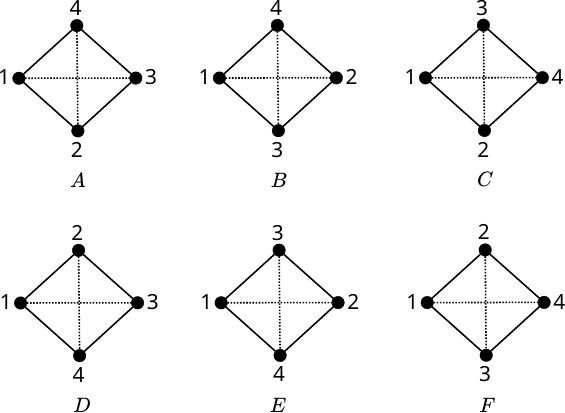}
    \caption{Possible quadrilaterals in the plane up to rotations; the numbers near the vertices correspond to the values of the variables, without loss of generality.}
    \label{fig:quadrilaterials}
\end{figure}

To a generator $g^{H_t}_{H_p}\in \overset{\mathllap{abc}}{\mathcal{G}^{4}_{n+3}}$, we assign a linear operator $\gamma_{\mathbb{C}^{\mathbf{T}_t},\mathbb{C}^{\mathbf{T}_p}}:\mathbb{C}^{\mathbf{T}_t}\rightarrow \mathbb{C}^{\mathbf{T}_p}$ defined by
\begin{gather}
\gamma(\Delta)=\Delta \mbox{ for all } \Delta\in \mathbf{T}_t\cap \mathbf{T}_p, \label{eq:operator-1}\\
\gamma(\Delta_{ijk})=\cos{\phi_{ijkl}}\Delta_{ijl} - \sin{\phi_{ijkl}}\Delta_{jkl}, \label{eq:operator-2}\\
\gamma(\Delta_{ikl})=\sin{\phi_{ijkl}}\Delta_{ijl} + \cos{\phi_{ijkl}}\Delta_{jkl}, \label{eq:operator-3}
\end{gather}
where $\cos{\phi_{ijkl}}=(-1)^{s}\sqrt{\frac{\zeta_{il}\zeta_{jk}}{\zeta_{ik}\zeta_{jl}}}$ and $\sin{\phi_{ijkl}}=(-1)^{r}\sqrt{\frac{\zeta_{ij}\zeta_{kl}}{\zeta_{ik}\zeta_{jl}}}$,
$$
s=\begin{cases} 1\ \text{for quadrilateral B,} \\ 0\ \text{otherwise,} \end{cases}
$$
and
$$
r=\begin{cases} 1\ \text{for quadrilateral C,} \\ 0\ \text{otherwise,} \end{cases}
$$
$\zeta_{uv}=\zeta_u-\zeta_v$ $(u,v \in \{1,\dots,n+3\})$, and all square roots are taken to be positive.

A generator $g^{H_t}_{H_p}\in \overset{\mathllap{abc}}{\mathcal{G}^{4}_{n+3}}$ uniquely determines the corresponding flip. If it is clear from the context which triangles are replaced by the flip $ik\rightarrow jl$, then the matrix of the corresponding linear map will be denoted by $A_{ikjl}$. This matrix has the following form:
\begin{equation}
\label{eq:ik_jl_matrix}
A_{ikjl} = \
    \begin{pmatrix}
* & \dots & * & \dots & * & \dots & * \\
& & & \vdots & & & \\
* & \dots & \cos{\phi_{ijkl}} & \dots & \sin{\phi_{ijkl}} & \dots & * \\
& & & \vdots & & & \\
* & \dots & -\sin{\phi_{ijkl}} & \dots & \cos{\phi_{ijkl}} & \dots & * \\
& & & \vdots & & & \\
* & \dots & * & \dots & * & \dots & *
    \end{pmatrix},
\end{equation}
where the symbol $*$ is equal either to $0$ or to $1$, depending on the positions of the triangles in the bases $\mathbf{f}_t$ and $\mathbf{f}_p$.

If the element $g^{H_t}_{H_p}\in \overset{\mathllap{abc}}{\mathcal{G}^{4}_{n+3}}$ corresponds to the flip $ik\rightarrow jl$, then the element $(g^{H_t}_{H_p})^{-1}\in \overset{\mathllap{abc}}{\mathcal{G}^{4}_{n+3}}$ corresponds to the flip $jl\rightarrow ik$. We therefore assign to it the linear map with matrix $A_{jlik}$, equal to the inverse matrix defined by formula (\ref{eq:ik_jl_matrix}):
\begin{equation}
\label{eq:jl_ik_matrix}
A_{jlik} = A_{ikjl}^{-1}.
\end{equation}

\begin{remark}\label{rem:alg_2}
In this case, the algorithm for determining the indices $i,j,k,l$ of the points in an arbitrary quadrilateral depends on the values of the variables associated with its vertices. We shall always list the vertices of the quadrilateral counterclockwise. The vertex with index $i$ is taken to be the vertex whose variable has the smallest value. The vertex $k$ is then determined automatically by the diagonal $ik$. The vertex $j$ is the vertex with the smaller variable among the two remaining vertices and the vertex $l$ is the remaining vertex. For example, suppose that a quadrilateral with vertices $(\zeta_s,\zeta_t,\zeta_u,\zeta_v)$ is placed in the plane, and that $\zeta_u < \zeta_s < \zeta_v < \zeta_t$. Then the vertex $i$ is the vertex with variable $\zeta_u$, the vertex $k$ is the vertex with variable $\zeta_s$, the vertex $j$ is the vertex with variable $\zeta_v$, and the vertex $l$ is the vertex with variable $\zeta_t$. Such a quadrilateral corresponds to the quadrilateral shown in Fig. \ref{fig:quadrilaterials}, B. For the quadrilateral $(\zeta_v,\zeta_u,\zeta_t,\zeta_s)$, the vertices have the same indices, but the embedding of this quadrilateral corresponds to the quadrilateral in Fig. \ref{fig:quadrilaterials}, E. Therefore, the operators defined in formulas (\ref{eq:operator-2})--(\ref{eq:operator-3}) differ in the first and second cases.
\end{remark}

By analogy with the map (\ref{eq:h}), we construct a map $f'_{n+3}:\overset{\mathllap{abc}}{\mathcal{G}^{4}_{n+3}} \rightarrow \text{GL}_{2n+1}(\mathbb{C})$ according to the rule:
\begin{equation}\label{eq:wh}
{f}'_{n+3}(g^{H_p}_{H_q}g^{H_q}_{H_r}\dots g^{H_t}_{H_u}g^{H_u}_{H_v}) = A^{H_u}_{H_v}A^{H_t}_{H_u}\cdots A^{H_q}_{H_r}A^{H_p}_{H_q}.
\end{equation}

\begin{theorem}
The map ${f}'_{n+3}$ is a homomorphism.
\end{theorem}

\begin{proof}
Using the algorithm in Remark \ref{rem:alg_2} and formula (\ref{eq:ik_jl_matrix}), we define the map ${f}'_{n+3}$ on the generators of the groupoid $\overset{\mathllap{abc}}{\mathcal{G}^{4}_{n+3}}$. The identity element is assigned the identity matrix $I_{2n+1}$. To the elements inverse to the generators, we assign the inverse matrices according to formula (\ref{eq:jl_ik_matrix}).

We prove that relations \ref{rel:identity}--\ref{rel:far_commute} hold for the matrices.

Relation \ref{rel:identity} is trivial, and relation \ref{rel:4_cycle} holds by construction.

Relation \ref{rel:5_cycle} for matrices of size $3\times 3$ is proved by computer enumeration and direct computation\footnote{It is sufficient to enumerate only those permutations that do not represent the same cycle; there are $\frac{(5-1)!}{2}=12$ such permutations. The matrix computations can be carried out using the computer algebra systems Sage and Maple.}. It follows from formulas (\ref{eq:operator-1})--(\ref{eq:operator-3}) that the nonzero coordinates of the vectors in the new basis depend only on the quadrilateral in which the flip occurs. Consequently, adding new vectors to the bases only enlarges the matrices by identity entries and does not affect the coordinates of the images of those vectors that participate in this relation. Thus, this pentagon relation holds for matrices of arbitrary size.

The proof of relation \ref{rel:far_commute} is analogous. The images of the vectors associated with quadrilaterals having at most one common edge do not depend on each other, nor do they depend on the presence or absence of other triangles in the triangulations, except for those participating in the flips from this relation. Therefore, it is sufficient to verify this relation by direct computation for triangulations consisting of two quadrilaterals; see Fig. \ref{fig:event_commutativity}.
\end{proof}

\subsection{Properties of the maps $f_{n+3}$ and ${f}'_{n+3}$}

The main difference between the maps $f_{n+3}$ and ${f}'_{n+3}$ is that the image of ${f}'_{n+3}$ consists of orthogonal matrices, whereas the matrices in the image of $f_{n+3}$ are not orthogonal.

\begin{figure}[h]
    \centering
    \includegraphics[width=0.15\textwidth]{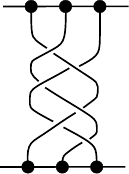}
    \caption{A braid whose closure is equivalent to the Borromean rings.}
    \label{fig:borromeo_link}
\end{figure}

The matrix invariants obtained using the maps $f_{n+3}$ and ${f}'_{n+3}$ are incomplete. As an example, one can compute the invariants of the braid shown in Fig. \ref{fig:borromeo_link} and verify that both invariants yield the identity matrix.

\section{Further directions}

The next step in this area is to construct a representation of the braid group $B_n$. To do this, it is necessary to modify the matrices obtained above in such a way that a single matrix corresponds to each generator $\sigma_i \in B_n$.

In \cite{Rohozhkin}, an idea for constructing an invariant of knots and links was also described as a further direction of research. However, the matrices corresponding to pure braids and obtained as products of the matrices defined by formulas (\ref{equation:matrix_n_2}), (\ref{eq:ik_jl_matrix}), and (\ref{eq:jl_ik_matrix}) have ones on the main diagonal and their determinants are always equal to $1$. Because of these properties, at the time of writing this paper, it has not been possible to obtain nontrivial invariants of knots and links. It may be that, in order to obtain a nontrivial result, one needs to deform the matrices described in this paper and consider the traces of the matrices corresponding to knots and links.

Another promising direction may be an attempt to describe the entire class of linear maps considered in this paper.

\section{Acknowledgments}

The author expresses special gratitude to Alexey V. Sleptsov for his indispensable assistance in writing this paper.

The author is also grateful to Igor Germanovich Korepanov, Igor M. Nikonov, Vasily O. Manturov, Kim Seongjeong, and Louis Kauffman for their constructive comments and discussions on this topic.

The author thanks Yuliia for her moral support during the preparation of this article.

\begin{appendices}
\section{Pentagon computation}\label{appendix:pentagon}

We construct the linear maps corresponding to the transformations of triangulations of a pentagon, as shown in Fig. \ref{fig:pentagon_korep_flips}. To the point with index $i$ we assign the variable $\zeta_i$.

\begin{figure}[h]
    \centering
    \includegraphics[width=0.6\textwidth]{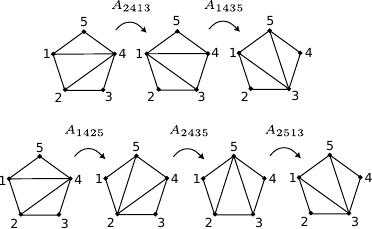}
    \caption{Transformations of a triangulation of a pentagon.}
    \label{fig:pentagon_korep_flips}
\end{figure}

$$
\gamma{(\Delta_{124},\Delta_{145},\Delta_{234})} = (\Delta_{123},\Delta_{134},\Delta_{145})
\left(\begin{array}{rrr}
\sqrt{{ \frac{{\left(\zeta_{1} - \zeta_{4}\right)} {\left(\zeta_{2} - \zeta_{3}\right)}}{{\left(\zeta_{1} - \zeta_{3}\right)} {\left(\zeta_{2} - \zeta_{4}\right)}} }} & 0 & -\sqrt{{ \frac{{\left(\zeta_{1} - \zeta_{2}\right)} {\left(\zeta_{3} - \zeta_{4}\right)}}{{\left(\zeta_{1} - \zeta_{3}\right)} {\left(\zeta_{2} - \zeta_{4}\right)}} }} \\
\sqrt{{ \frac{{\left(\zeta_{1} - \zeta_{2}\right)} {\left(\zeta_{3} - \zeta_{4}\right)}}{{\left(\zeta_{1} - \zeta_{3}\right)} {\left(\zeta_{2} - \zeta_{4}\right)}} }} & 0 & \sqrt{{ \frac{{\left(\zeta_{1} - \zeta_{4}\right)} {\left(\zeta_{2} - \zeta_{3}\right)}}{{\left(\zeta_{1} - \zeta_{3}\right)} {\left(\zeta_{2} - \zeta_{4}\right)}} }} \\
0 & 1 & 0
\end{array}\right),
$$

$$
\gamma{(\Delta_{123},\Delta_{134},\Delta_{145})} = (\Delta_{123},\Delta_{135},\Delta_{345})
\left(\begin{array}{rrr}
1 & 0 & 0 \\
0 & \sqrt{{ \frac{{\left(\zeta_{1} - \zeta_{5}\right)} {\left(\zeta_{3} - \zeta_{4}\right)}}{{\left(\zeta_{1} - \zeta_{4}\right)} {\left(\zeta_{3} - \zeta_{5}\right)}} }} & \sqrt{{ \frac{{\left(\zeta_{1} - \zeta_{3}\right)} {\left(\zeta_{4} - \zeta_{5}\right)}}{{\left(\zeta_{1} - \zeta_{4}\right)} {\left(\zeta_{3} - \zeta_{5}\right)}} }} \\
0 & -\sqrt{{ \frac{{\left(\zeta_{1} - \zeta_{3}\right)} {\left(\zeta_{4} - \zeta_{5}\right)}}{{\left(\zeta_{1} - \zeta_{4}\right)} {\left(\zeta_{3} - \zeta_{5}\right)}} }} & \sqrt{{ \frac{{\left(\zeta_{1} - \zeta_{5}\right)} {\left(\zeta_{3} - \zeta_{4}\right)}}{{\left(\zeta_{1} - \zeta_{4}\right)} {\left(\zeta_{3} - \zeta_{5}\right)}} }}
\end{array}\right),
$$

$$
\gamma{(\Delta_{124},\Delta_{145},\Delta_{234})} = (\Delta_{125},\Delta_{234},\Delta_{245})
\left(\begin{array}{rrr}
\sqrt{{ \frac{{\left(\zeta_{1} - \zeta_{5}\right)} {\left(\zeta_{2} - \zeta_{4}\right)}}{{\left(\zeta_{1} - \zeta_{4}\right)} {\left(\zeta_{2} - \zeta_{5}\right)}} }} & \sqrt{{ \frac{{\left(\zeta_{1} - \zeta_{2}\right)} {\left(\zeta_{4} - \zeta_{5}\right)}}{{\left(\zeta_{1} - \zeta_{4}\right)} {\left(\zeta_{2} - \zeta_{5}\right)}} }} & 0 \\
0 & 0 & 1 \\
-\sqrt{{ \frac{{\left(\zeta_{1} - \zeta_{2}\right)} {\left(\zeta_{4} - \zeta_{5}\right)}}{{\left(\zeta_{1} - \zeta_{4}\right)} {\left(\zeta_{2} - \zeta_{5}\right)}} }} & \sqrt{{ \frac{{\left(\zeta_{1} - \zeta_{5}\right)} {\left(\zeta_{2} - \zeta_{4}\right)}}{{\left(\zeta_{1} - \zeta_{4}\right)} {\left(\zeta_{2} - \zeta_{5}\right)}} }} & 0
\end{array}\right),
$$

$$
\gamma{(\Delta_{125},\Delta_{234},\Delta_{245})} = (\Delta_{125},\Delta_{235},\Delta_{345})
\left(\begin{array}{rrr}
1 & 0 & 0 \\
0 & \sqrt{{ \frac{{\left(\zeta_{2} - \zeta_{5}\right)} {\left(\zeta_{3} - \zeta_{4}\right)}}{{\left(\zeta_{2} - \zeta_{4}\right)} {\left(\zeta_{3} - \zeta_{5}\right)}} }} & \sqrt{{ \frac{{\left(\zeta_{2} - \zeta_{3}\right)} {\left(\zeta_{4} - \zeta_{5}\right)}}{{\left(\zeta_{2} - \zeta_{4}\right)} {\left(\zeta_{3} - \zeta_{5}\right)}} }} \\
0 & -\sqrt{{ \frac{{\left(\zeta_{2} - \zeta_{3}\right)} {\left(\zeta_{4} - \zeta_{5}\right)}}{{\left(\zeta_{2} - \zeta_{4}\right)} {\left(\zeta_{3} - \zeta_{5}\right)}} }} & \sqrt{{ \frac{{\left(\zeta_{2} - \zeta_{5}\right)} {\left(\zeta_{3} - \zeta_{4}\right)}}{{\left(\zeta_{2} - \zeta_{4}\right)} {\left(\zeta_{3} - \zeta_{5}\right)}} }}
\end{array}\right),
$$

$$
\gamma{(\Delta_{125},\Delta_{235},\Delta_{345})} = (\Delta_{123},\Delta_{135},\Delta_{345})
\left(\begin{array}{rrr}
\sqrt{{ \frac{{\left(\zeta_{1} - \zeta_{5}\right)} {\left(\zeta_{2} - \zeta_{3}\right)}}{{\left(\zeta_{1} - \zeta_{3}\right)} {\left(\zeta_{2} - \zeta_{5}\right)}} }} & -\sqrt{{ \frac{{\left(\zeta_{1} - \zeta_{2}\right)} {\left(\zeta_{3} - \zeta_{5}\right)}}{{\left(\zeta_{1} - \zeta_{3}\right)} {\left(\zeta_{2} - \zeta_{5}\right)}} }} & 0 \\
\sqrt{{ \frac{{\left(\zeta_{1} - \zeta_{2}\right)} {\left(\zeta_{3} - \zeta_{5}\right)}}{{\left(\zeta_{1} - \zeta_{3}\right)} {\left(\zeta_{2} - \zeta_{5}\right)}} }} & \sqrt{{ \frac{{\left(\zeta_{1} - \zeta_{5}\right)} {\left(\zeta_{2} - \zeta_{3}\right)}}{{\left(\zeta_{1} - \zeta_{3}\right)} {\left(\zeta_{2} - \zeta_{5}\right)}} }} & 0 \\
0 & 0 & 1
\end{array}\right).
$$

The relation for this pentagon has the following form:

\begin{gather*}
\left(\begin{array}{rrr}
1 & 0 & 0 \\
0 & \sqrt{{ \frac{{\left(\zeta_{1} - \zeta_{5}\right)} {\left(\zeta_{3} - \zeta_{4}\right)}}{{\left(\zeta_{1} - \zeta_{4}\right)} {\left(\zeta_{3} - \zeta_{5}\right)}} }} & \sqrt{{ \frac{{\left(\zeta_{1} - \zeta_{3}\right)} {\left(\zeta_{4} - \zeta_{5}\right)}}{{\left(\zeta_{1} - \zeta_{4}\right)} {\left(\zeta_{3} - \zeta_{5}\right)}} }} \\
0 & -\sqrt{{ \frac{{\left(\zeta_{1} - \zeta_{3}\right)} {\left(\zeta_{4} - \zeta_{5}\right)}}{{\left(\zeta_{1} - \zeta_{4}\right)} {\left(\zeta_{3} - \zeta_{5}\right)}} }} & \sqrt{{ \frac{{\left(\zeta_{1} - \zeta_{5}\right)} {\left(\zeta_{3} - \zeta_{4}\right)}}{{\left(\zeta_{1} - \zeta_{4}\right)} {\left(\zeta_{3} - \zeta_{5}\right)}} }}
\end{array}\right)
\left(\begin{array}{rrr}
\sqrt{{ \frac{{\left(\zeta_{1} - \zeta_{4}\right)} {\left(\zeta_{2} - \zeta_{3}\right)}}{{\left(\zeta_{1} - \zeta_{3}\right)} {\left(\zeta_{2} - \zeta_{4}\right)}} }} & 0 & -\sqrt{{ \frac{{\left(\zeta_{1} - \zeta_{2}\right)} {\left(\zeta_{3} - \zeta_{4}\right)}}{{\left(\zeta_{1} - \zeta_{3}\right)} {\left(\zeta_{2} - \zeta_{4}\right)}} }} \\
\sqrt{{ \frac{{\left(\zeta_{1} - \zeta_{2}\right)} {\left(\zeta_{3} - \zeta_{4}\right)}}{{\left(\zeta_{1} - \zeta_{3}\right)} {\left(\zeta_{2} - \zeta_{4}\right)}} }} & 0 & \sqrt{{ \frac{{\left(\zeta_{1} - \zeta_{4}\right)} {\left(\zeta_{2} - \zeta_{3}\right)}}{{\left(\zeta_{1} - \zeta_{3}\right)} {\left(\zeta_{2} - \zeta_{4}\right)}} }} \\
0 & 1 & 0
\end{array}\right) \\
= \left(\begin{array}{rrr}
\sqrt{{ \frac{{\left(\zeta_{1} - \zeta_{5}\right)} {\left(\zeta_{2} - \zeta_{3}\right)}}{{\left(\zeta_{1} - \zeta_{3}\right)} {\left(\zeta_{2} - \zeta_{5}\right)}} }} & -\sqrt{{ \frac{{\left(\zeta_{1} - \zeta_{2}\right)} {\left(\zeta_{3} - \zeta_{5}\right)}}{{\left(\zeta_{1} - \zeta_{3}\right)} {\left(\zeta_{2} - \zeta_{5}\right)}} }} & 0 \\
\sqrt{{ \frac{{\left(\zeta_{1} - \zeta_{2}\right)} {\left(\zeta_{3} - \zeta_{5}\right)}}{{\left(\zeta_{1} - \zeta_{3}\right)} {\left(\zeta_{2} - \zeta_{5}\right)}} }} & \sqrt{{ \frac{{\left(\zeta_{1} - \zeta_{5}\right)} {\left(\zeta_{2} - \zeta_{3}\right)}}{{\left(\zeta_{1} - \zeta_{3}\right)} {\left(\zeta_{2} - \zeta_{5}\right)}} }} & 0 \\
0 & 0 & 1
\end{array}\right)
\left(\begin{array}{rrr}
1 & 0 & 0 \\
0 & \sqrt{{ \frac{{\left(\zeta_{2} - \zeta_{5}\right)} {\left(\zeta_{3} - \zeta_{4}\right)}}{{\left(\zeta_{2} - \zeta_{4}\right)} {\left(\zeta_{3} - \zeta_{5}\right)}} }} & \sqrt{{ \frac{{\left(\zeta_{2} - \zeta_{3}\right)} {\left(\zeta_{4} - \zeta_{5}\right)}}{{\left(\zeta_{2} - \zeta_{4}\right)} {\left(\zeta_{3} - \zeta_{5}\right)}} }} \\
0 & -\sqrt{{ \frac{{\left(\zeta_{2} - \zeta_{3}\right)} {\left(\zeta_{4} - \zeta_{5}\right)}}{{\left(\zeta_{2} - \zeta_{4}\right)} {\left(\zeta_{3} - \zeta_{5}\right)}} }} & \sqrt{{ \frac{{\left(\zeta_{2} - \zeta_{5}\right)} {\left(\zeta_{3} - \zeta_{4}\right)}}{{\left(\zeta_{2} - \zeta_{4}\right)} {\left(\zeta_{3} - \zeta_{5}\right)}} }}
\end{array}\right) \\
\cdot \left(\begin{array}{rrr}
\sqrt{{ \frac{{\left(\zeta_{1} - \zeta_{5}\right)} {\left(\zeta_{2} - \zeta_{4}\right)}}{{\left(\zeta_{1} - \zeta_{4}\right)} {\left(\zeta_{2} - \zeta_{5}\right)}} }} & \sqrt{{ \frac{{\left(\zeta_{1} - \zeta_{2}\right)} {\left(\zeta_{4} - \zeta_{5}\right)}}{{\left(\zeta_{1} - \zeta_{4}\right)} {\left(\zeta_{2} - \zeta_{5}\right)}} }} & 0 \\
0 & 0 & 1 \\
-\sqrt{{ \frac{{\left(\zeta_{1} - \zeta_{2}\right)} {\left(\zeta_{4} - \zeta_{5}\right)}}{{\left(\zeta_{1} - \zeta_{4}\right)} {\left(\zeta_{2} - \zeta_{5}\right)}} }} & \sqrt{{ \frac{{\left(\zeta_{1} - \zeta_{5}\right)} {\left(\zeta_{2} - \zeta_{4}\right)}}{{\left(\zeta_{1} - \zeta_{4}\right)} {\left(\zeta_{2} - \zeta_{5}\right)}} }} & 0
\end{array}\right).
\end{gather*}

\section{Computation of pure braids}

We give an example of computing pure braids from the relation of the pure braid group \cite{Manturov-Knots}:
$$
b_{ij}b_{kl}=b_{kl}b_{ij}.
$$

Figure \ref{fig:pure_braid_calc_example} shows the dynamical system corresponding to the case $i<k<l<j$.

\begin{figure}[h]
    \centering
    \includegraphics[width=0.8\textwidth]{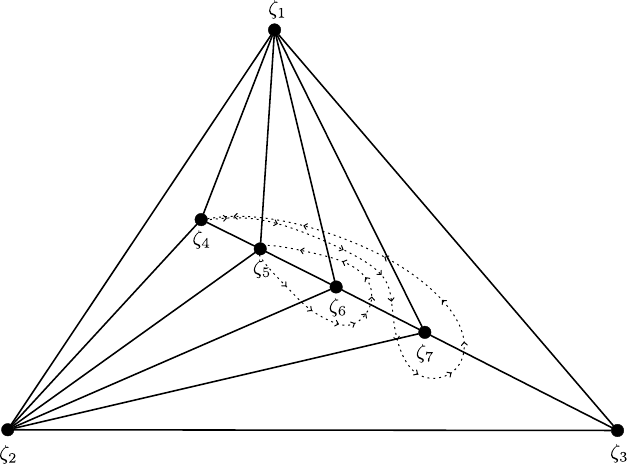}
    \caption{An example of a dynamical system; the arrows indicate the directions of motion of the point $\zeta_4$ of the braid $b_{ij}=b_{\zeta_4\zeta_7}$ and the point $\zeta_5$ of the braid $b_{kl}=b_{\zeta_5\zeta_6}$.}
    \label{fig:pure_braid_calc_example}
\end{figure}

In what follows, to simplify the computations, we assume that $\zeta_i=i$. We shall also round the entries of the matrices to several decimal places for ease of presentation.

The following product of matrices corresponds to the braid $b_{\zeta_4\zeta_7}$:
\begin{gather*}
f'_{n+3}(b_{\zeta_4\zeta_7}) = A_{b_{ij}} = A_{4615}A_{1524}A_{4716}A_{3417}A_{1645}A_{1746}\\ A_{2437}A_{3714}A_{4627}A_{2734}A_{1467}A_{6724}A_{4516}A_{1647}A_{2415}A_{1546}.
\end{gather*}

In matrix form, this product has the form:

\begin{gather*}
A_{b_{ij}} = 
\left(\begin{array}{rrrrrrrrr}
1.0 & 0.0 & 0.0 & 0.0 & 0.0 & 0.0 & 0.0 & 0.0 & 0.0 \\
0.0 & 1.0 & 0.0 & 0.0 & 0.0 & 0.0 & 0.0 & 0.0 & 0.0 \\
0.0 & 0.0 & 0.79056942 & 0.0 & 0.0 & 0.0 & 0.0 & 0.0 & -0.61237244 \\
0.0 & 0.0 & 0.61237244 & 0.0 & 0.0 & 0.0 & 0.0 & 0.0 & 0.79056942 \\
0.0 & 0.0 & 0.0 & 1.0 & 0.0 & 0.0 & 0.0 & 0.0 & 0.0 \\
0.0 & 0.0 & 0.0 & 0.0 & 1.0 & 0.0 & 0.0 & 0.0 & 0.0 \\
0.0 & 0.0 & 0.0 & 0.0 & 0.0 & 1.0 & 0.0 & 0.0 & 0.0 \\
0.0 & 0.0 & 0.0 & 0.0 & 0.0 & 0.0 & 1.0 & 0.0 & 0.0 \\
0.0 & 0.0 & 0.0 & 0.0 & 0.0 & 0.0 & 0.0 & 1.0 & 0.0
\end{array}\right) \\
\cdot \left(\begin{array}{rrrrrrrrr}
1.06066017 & 0.0 & -0.35355339i & 0.0 & 0.0 & 0.0 & 0.0 & 0.0 & 0.0 \\
0.0 & 1.0 & 0.0 & 0.0 & 0.0 & 0.0 & 0.0 & 0.0 & 0.0 \\
0.0 & 0.0 & 0.0 & 1.0 & 0.0 & 0.0 & 0.0 & 0.0 & 0.0 \\
0.0 & 0.0 & 0.0 & 0.0 & 1.0 & 0.0 & 0.0 & 0.0 & 0.0 \\
0.0 & 0.0 & 0.0 & 0.0 & 0.0 & 1.0 & 0.0 & 0.0 & 0.0 \\
0.35355339i & 0.0 & 1.06066017 & 0.0 & 0.0 & 0.0 & 0.0 & 0.0 & 0.0 \\
0.0 & 0.0 & 0.0 & 0.0 & 0.0 & 0.0 & 1.0 & 0.0 & 0.0 \\
0.0 & 0.0 & 0.0 & 0.0 & 0.0 & 0.0 & 0.0 & 1.0 & 0.0 \\
0.0 & 0.0 & 0.0 & 0.0 & 0.0 & 0.0 & 0.0 & 0.0 & 1.0
\end{array}\right)
\end{gather*}

\begin{gather*}
\cdot \left(\begin{array}{rrrrrrrrr}
1.0 & 0.0 & 0.0 & 0.0 & 0.0 & 0.0 & 0.0 & 0.0 & 0.0 \\
0.0 & 1.0 & 0.0 & 0.0 & 0.0 & 0.0 & 0.0 & 0.0 & 0.0 \\
0.0 & 0.0 & 1.0 & 0.0 & 0.0 & 0.0 & 0.0 & 0.0 & 0.0 \\
0.0 & 0.0 & 0.0 & 0.89442719 & 0.0 & 0.0 & 0.0 & 0.0 & -0.4472136 \\
0.0 & 0.0 & 0.0 & 0.4472136 & 0.0 & 0.0 & 0.0 & 0.0 & 0.89442719 \\
0.0 & 0.0 & 0.0 & 0.0 & 1.0 & 0.0 & 0.0 & 0.0 & 0.0 \\
0.0 & 0.0 & 0.0 & 0.0 & 0.0 & 1.0 & 0.0 & 0.0 & 0.0 \\
0.0 & 0.0 & 0.0 & 0.0 & 0.0 & 0.0 & 1.0 & 0.0 & 0.0 \\
0.0 & 0.0 & 0.0 & 0.0 & 0.0 & 0.0 & 0.0 & 1.0 & 0.0
\end{array}\right) \\
\cdot \left(\begin{array}{rrrrrrrrr}
1.0 & 0.0 & 0.0 & 0.0 & 0.0 & 0.0 & 0.0 & 0.0 & 0.0 \\
0.0 & 1.41421356 & 0.0 & 0.0 & 0.0 & 0.0 & -1.0i & 0.0 & 0.0 \\
0.0 & 0.0 & 1.0 & 0.0 & 0.0 & 0.0 & 0.0 & 0.0 & 0.0 \\
0.0 & 1.0i & 0.0 & 0.0 & 0.0 & 0.0 & 1.41421356 & 0.0 & 0.0 \\
0.0 & 0.0 & 0.0 & 1.0 & 0.0 & 0.0 & 0.0 & 0.0 & 0.0 \\
0.0 & 0.0 & 0.0 & 0.0 & 1.0 & 0.0 & 0.0 & 0.0 & 0.0 \\
0.0 & 0.0 & 0.0 & 0.0 & 0.0 & 1.0 & 0.0 & 0.0 & 0.0 \\
0.0 & 0.0 & 0.0 & 0.0 & 0.0 & 0.0 & 0.0 & 1.0 & 0.0 \\
0.0 & 0.0 & 0.0 & 0.0 & 0.0 & 0.0 & 0.0 & 0.0 & 1.0
\end{array}\right) \\
\cdot \left(\begin{array}{rrrrrrrrr}
1.0 & 0.0 & 0.0 & 0.0 & 0.0 & 0.0 & 0.0 & 0.0 & 0.0 \\
0.0 & 1.0 & 0.0 & 0.0 & 0.0 & 0.0 & 0.0 & 0.0 & 0.0 \\
0.0 & 0.0 & 1.26491106 & 0.77459667i & 0.0 & 0.0 & 0.0 & 0.0 & 0.0 \\
0.0 & 0.0 & 0.0 & 0.0 & 1.0 & 0.0 & 0.0 & 0.0 & 0.0 \\
0.0 & 0.0 & 0.0 & 0.0 & 0.0 & 1.0 & 0.0 & 0.0 & 0.0 \\
0.0 & 0.0 & 0.0 & 0.0 & 0.0 & 0.0 & 1.0 & 0.0 & 0.0 \\
0.0 & 0.0 & 0.0 & 0.0 & 0.0 & 0.0 & 0.0 & 1.0 & 0.0 \\
0.0 & 0.0 & -0.77459667i & 1.26491106 & 0.0 & 0.0 & 0.0 & 0.0 & 0.0 \\
0.0 & 0.0 & 0.0 & 0.0 & 0.0 & 0.0 & 0.0 & 0.0 & 1.0
\end{array}\right) \\
\cdot \left(\begin{array}{rrrrrrrrr}
1.0 & 0.0 & 0.0 & 0.0 & 0.0 & 0.0 & 0.0 & 0.0 & 0.0 \\
0.0 & 1.0 & 0.0 & 0.0 & 0.0 & 0.0 & 0.0 & 0.0 & 0.0 \\
0.0 & 0.0 & 1.11803399 & 0.0 & 0.5i & 0.0 & 0.0 & 0.0 & 0.0 \\
0.0 & 0.0 & 0.0 & 1.0 & 0.0 & 0.0 & 0.0 & 0.0 & 0.0 \\
0.0 & 0.0 & 0.0 & 0.0 & 0.0 & 1.0 & 0.0 & 0.0 & 0.0 \\
0.0 & 0.0 & 0.0 & 0.0 & 0.0 & 0.0 & 1.0 & 0.0 & 0.0 \\
0.0 & 0.0 & 0.0 & 0.0 & 0.0 & 0.0 & 0.0 & 1.0 & 0.0 \\
0.0 & 0.0 & 0.0 & 0.0 & 0.0 & 0.0 & 0.0 & 0.0 & 1.0 \\
0.0 & 0.0 & -0.5i & 0.0 & 1.11803399 & 0.0 & 0.0 & 0.0 & 0.0
\end{array}\right)
\end{gather*}

\begin{gather*}
\cdot \left(\begin{array}{rrrrrrrrr}
1.0 & 0.0 & 0.0 & 0.0 & 0.0 & 0.0 & 0.0 & 0.0 & 0.0 \\
0.0 & 1.0 & 0.0 & 0.0 & 0.0 & 0.0 & 0.0 & 0.0 & 0.0 \\
0.0 & 0.0 & 1.0 & 0.0 & 0.0 & 0.0 & 0.0 & 0.0 & 0.0 \\
0.0 & 0.0 & 0.0 & 1.0 & 0.0 & 0.0 & 0.0 & 0.0 & 0.0 \\
0.0 & 0.0 & 0.0 & 0.0 & 1.0 & 0.0 & 0.0 & 0.0 & 0.0 \\
0.0 & 0.0 & 0.0 & 0.0 & 0.0 & 0.79056942 & 0.61237244 & 0.0 & 0.0 \\
0.0 & 0.0 & 0.0 & 0.0 & 0.0 & 0.0 & 0.0 & 1.0 & 0.0 \\
0.0 & 0.0 & 0.0 & 0.0 & 0.0 & 0.0 & 0.0 & 0.0 & 1.0 \\
0.0 & 0.0 & 0.0 & 0.0 & 0.0 & -0.61237244 & 0.79056942 & 0.0 & 0.0
\end{array}\right) \\
\cdot \left(\begin{array}{rrrrrrrrr}
1.0 & 0.0 & 0.0 & 0.0 & 0.0 & 0.0 & 0.0 & 0.0 & 0.0 \\
0.0 & 0.70710678 & 0.0 & 0.0 & 0.0 & 0.0 & 0.0 & 0.0 & -0.70710678 \\
0.0 & 0.70710678 & 0.0 & 0.0 & 0.0 & 0.0 & 0.0 & 0.0 & 0.70710678 \\
0.0 & 0.0 & 1.0 & 0.0 & 0.0 & 0.0 & 0.0 & 0.0 & 0.0 \\
0.0 & 0.0 & 0.0 & 1.0 & 0.0 & 0.0 & 0.0 & 0.0 & 0.0 \\
0.0 & 0.0 & 0.0 & 0.0 & 1.0 & 0.0 & 0.0 & 0.0 & 0.0 \\
0.0 & 0.0 & 0.0 & 0.0 & 0.0 & 1.0 & 0.0 & 0.0 & 0.0 \\
0.0 & 0.0 & 0.0 & 0.0 & 0.0 & 0.0 & 1.0 & 0.0 & 0.0 \\
0.0 & 0.0 & 0.0 & 0.0 & 0.0 & 0.0 & 0.0 & 1.0 & 0.0
\end{array}\right) \\
\cdot \left(\begin{array}{rrrrrrrrr}
1.0 & 0.0 & 0.0 & 0.0 & 0.0 & 0.0 & 0.0 & 0.0 & 0.0 \\
0.0 & 1.0 & 0.0 & 0.0 & 0.0 & 0.0 & 0.0 & 0.0 & 0.0 \\
0.0 & 0.0 & 1.0 & 0.0 & 0.0 & 0.0 & 0.0 & 0.0 & 0.0 \\
0.0 & 0.0 & 0.0 & 1.0 & 0.0 & 0.0 & 0.0 & 0.0 & 0.0 \\
0.0 & 0.0 & 0.0 & 0.0 & 1.0 & 0.0 & 0.0 & 0.0 & 0.0 \\
0.0 & 0.0 & 0.0 & 0.0 & 0.0 & 1.09544512 & 0.0 & 0.0 & 0.4472136i \\
0.0 & 0.0 & 0.0 & 0.0 & 0.0 & 0.0 & 1.0 & 0.0 & 0.0 \\
0.0 & 0.0 & 0.0 & 0.0 & 0.0 & -0.4472136i & 0.0 & 0.0 & 1.09544512 \\
0.0 & 0.0 & 0.0 & 0.0 & 0.0 & 0.0 & 0.0 & 1.0 & 0.0
\end{array}\right) \\
\cdot \left(\begin{array}{rrrrrrrrr}
1.0 & 0.0 & 0.0 & 0.0 & 0.0 & 0.0 & 0.0 & 0.0 & 0.0 \\
0.0 & 1.0 & 0.0 & 0.0 & 0.0 & 0.0 & 0.0 & 0.0 & 0.0 \\
0.0 & 0.0 & 1.0 & 0.0 & 0.0 & 0.0 & 0.0 & 0.0 & 0.0 \\
0.0 & 0.0 & 0.0 & 1.0 & 0.0 & 0.0 & 0.0 & 0.0 & 0.0 \\
0.0 & 0.0 & 0.0 & 0.0 & 1.26491106 & 0.0 & -0.77459667i & 0.0 & 0.0 \\
0.0 & 0.0 & 0.0 & 0.0 & 0.0 & 1.0 & 0.0 & 0.0 & 0.0 \\
0.0 & 0.0 & 0.0 & 0.0 & 0.0 & 0.0 & 0.0 & 1.0 & 0.0 \\
0.0 & 0.0 & 0.0 & 0.0 & 0.77459667i & 0.0 & 1.26491106 & 0.0 & 0.0 \\
0.0 & 0.0 & 0.0 & 0.0 & 0.0 & 0.0 & 0.0 & 0.0 & 1.0
\end{array}\right)
\end{gather*}

\begin{gather*}
\cdot \left(\begin{array}{rrrrrrrrr}
1.0 & 0.0 & 0.0 & 0.0 & 0.0 & 0.0 & 0.0 & 0.0 & 0.0 \\
0.0 & 1.0 & 0.0 & 0.0 & 0.0 & 0.0 & 0.0 & 0.0 & 0.0 \\
0.0 & 0.0 & 0.0 & 0.0 & 1.0 & 0.0 & 0.0 & 0.0 & 0.0 \\
0.0 & 0.0 & -2.0i & 2.23606798 & 0.0 & 0.0 & 0.0 & 0.0 & 0.0 \\
0.0 & 0.0 & 0.0 & 0.0 & 0.0 & 1.0 & 0.0 & 0.0 & 0.0 \\
0.0 & 0.0 & 0.0 & 0.0 & 0.0 & 0.0 & 1.0 & 0.0 & 0.0 \\
0.0 & 0.0 & 0.0 & 0.0 & 0.0 & 0.0 & 0.0 & 1.0 & 0.0 \\
0.0 & 0.0 & 0.0 & 0.0 & 0.0 & 0.0 & 0.0 & 0.0 & 1.0 \\
0.0 & 0.0 & -2.23606798 & -2.0i & 0.0 & 0.0 & 0.0 & 0.0 & 0.0
\end{array}\right) \\
\cdot \left(\begin{array}{rrrrrrrrr}
1.0 & 0.0 & 0.0 & 0.0 & 0.0 & 0.0 & 0.0 & 0.0 & 0.0 \\
0.0 & 1.0 & 0.0 & 0.0 & 0.0 & 0.0 & 0.0 & 0.0 & 0.0 \\
0.0 & 0.0 & 1.0 & 0.0 & 0.0 & 0.0 & 0.0 & 0.0 & 0.0 \\
0.0 & 0.0 & 0.0 & 1.0 & 0.0 & 0.0 & 0.0 & 0.0 & 0.0 \\
0.0 & 0.0 & 0.0 & 0.0 & 1.0 & 0.0 & 0.0 & 0.0 & 0.0 \\
0.0 & 0.0 & 0.0 & 0.0 & 0.0 & 1.0 & 0.0 & 0.0 & 0.0 \\
0.0 & 0.0 & 0.0 & 0.0 & 0.0 & 0.0 & 0.0 & 2.23606798i & -2.44948974 \\
0.0 & 0.0 & 0.0 & 0.0 & 0.0 & 0.0 & 0.0 & 2.44948974 & 2.23606798i \\
0.0 & 0.0 & 0.0 & 0.0 & 0.0 & 0.0 & 1.0 & 0.0 & 0.0
\end{array}\right) \\
\cdot \left(\begin{array}{rrrrrrrrr}
1.0 & 0.0 & 0.0 & 0.0 & 0.0 & 0.0 & 0.0 & 0.0 & 0.0 \\
0.0 & 1.0 & 0.0 & 0.0 & 0.0 & 0.0 & 0.0 & 0.0 & 0.0 \\
0.0 & 0.0 & 1.26491106 & 0.0 & 0.0 & 0.0 & 0.0 & -0.77459667i & 0.0 \\
0.0 & 0.0 & 0.0 & 1.0 & 0.0 & 0.0 & 0.0 & 0.0 & 0.0 \\
0.0 & 0.0 & 0.77459667i & 0.0 & 0.0 & 0.0 & 0.0 & 1.26491106 & 0.0 \\
0.0 & 0.0 & 0.0 & 0.0 & 1.0 & 0.0 & 0.0 & 0.0 & 0.0 \\
0.0 & 0.0 & 0.0 & 0.0 & 0.0 & 1.0 & 0.0 & 0.0 & 0.0 \\
0.0 & 0.0 & 0.0 & 0.0 & 0.0 & 0.0 & 1.0 & 0.0 & 0.0 \\
0.0 & 0.0 & 0.0 & 0.0 & 0.0 & 0.0 & 0.0 & 0.0 & 1.0
\end{array}\right) \\
\cdot \left(\begin{array}{rrrrrrrrr}
1.0 & 0.0 & 0.0 & 0.0 & 0.0 & 0.0 & 0.0 & 0.0 & 0.0 \\
0.0 & 1.0 & 0.0 & 0.0 & 0.0 & 0.0 & 0.0 & 0.0 & 0.0 \\
0.0 & 0.0 & 1.0 & 0.0 & 0.0 & 0.0 & 0.0 & 0.0 & 0.0 \\
0.0 & 0.0 & 0.0 & 0.89442719 & 0.4472136 & 0.0 & 0.0 & 0.0 & 0.0 \\
0.0 & 0.0 & 0.0 & 0.0 & 0.0 & 1.0 & 0.0 & 0.0 & 0.0 \\
0.0 & 0.0 & 0.0 & 0.0 & 0.0 & 0.0 & 1.0 & 0.0 & 0.0 \\
0.0 & 0.0 & 0.0 & 0.0 & 0.0 & 0.0 & 0.0 & 1.0 & 0.0 \\
0.0 & 0.0 & 0.0 & 0.0 & 0.0 & 0.0 & 0.0 & 0.0 & 1.0 \\
0.0 & 0.0 & 0.0 & -0.4472136 & 0.89442719 & 0.0 & 0.0 & 0.0 & 0.0
\end{array}\right)
\end{gather*}

\begin{gather*}
\cdot \left(\begin{array}{rrrrrrrrr}
1.06066017 & 0.0 & 0.0 & 0.0 & 0.0 & 0.35355339i & 0.0 & 0.0 & 0.0 \\
0.0 & 1.0 & 0.0 & 0.0 & 0.0 & 0.0 & 0.0 & 0.0 & 0.0 \\
-0.35355339i & 0.0 & 0.0 & 0.0 & 0.0 & 1.06066017 & 0.0 & 0.0 & 0.0 \\
0.0 & 0.0 & 1.0 & 0.0 & 0.0 & 0.0 & 0.0 & 0.0 & 0.0 \\
0.0 & 0.0 & 0.0 & 1.0 & 0.0 & 0.0 & 0.0 & 0.0 & 0.0 \\
0.0 & 0.0 & 0.0 & 0.0 & 1.0 & 0.0 & 0.0 & 0.0 & 0.0 \\
0.0 & 0.0 & 0.0 & 0.0 & 0.0 & 0.0 & 1.0 & 0.0 & 0.0 \\
0.0 & 0.0 & 0.0 & 0.0 & 0.0 & 0.0 & 0.0 & 1.0 & 0.0 \\
0.0 & 0.0 & 0.0 & 0.0 & 0.0 & 0.0 & 0.0 & 0.0 & 1.0
\end{array}\right) \\
\cdot \left(\begin{array}{rrrrrrrrr}
1.0 & 0.0 & 0.0 & 0.0 & 0.0 & 0.0 & 0.0 & 0.0 & 0.0 \\
0.0 & 1.0 & 0.0 & 0.0 & 0.0 & 0.0 & 0.0 & 0.0 & 0.0 \\
0.0 & 0.0 & 0.79056942 & 0.61237244 & 0.0 & 0.0 & 0.0 & 0.0 & 0.0 \\
0.0 & 0.0 & 0.0 & 0.0 & 1.0 & 0.0 & 0.0 & 0.0 & 0.0 \\
0.0 & 0.0 & 0.0 & 0.0 & 0.0 & 1.0 & 0.0 & 0.0 & 0.0 \\
0.0 & 0.0 & 0.0 & 0.0 & 0.0 & 0.0 & 1.0 & 0.0 & 0.0 \\
0.0 & 0.0 & 0.0 & 0.0 & 0.0 & 0.0 & 0.0 & 1.0 & 0.0 \\
0.0 & 0.0 & 0.0 & 0.0 & 0.0 & 0.0 & 0.0 & 0.0 & 1.0 \\
0.0 & 0.0 & -0.61237244 & 0.79056942 & 0.0 & 0.0 & 0.0 & 0.0 & 0.0
\end{array}\right) \\
= \left(\begin{array}{rrrrrrrrr}
1.0 & -0.35355 & 0.0 & 0.0 & 1.11803 & 0.27386 & 0.0 & 0.0 & -1.09545 \\
0.35355 & 1.0 & -1.0 & 0.0 & 0.0 & 0.0 & -1.06066 & 0.0 & 0.0 \\
0.0 & 1.0 & 1.0 & 0.0 & -3.16228 & -0.7746 & 0.0 & 0.0 & 3.09839 \\
0.0 & 0.0 & 0.0 & 1.0 & 0.0 & 0.0 & 0.0 & 0.0 & 0.0 \\
-1.11803 & 0.0 & 3.16228 & 0.0 & 1.0 & 0.0 & -3.3541 & 0.0 & 0.0 \\
-0.27386 & 0.0 & 0.7746 & 0.0 & 0.0 & 1.0 & -0.82158 & 0.0 & 0.0 \\
0.0 & 1.06066 & 0.0 & 0.0 & 3.3541 & 0.82158 & 1.0 & 0.0 & 3.28634 \\
0.0 & 0.0 & 0.0 & 0.0 & 0.0 & 0.0 & 0.0 & 1.0 & 0.0 \\
1.09545 & 0.0 & -3.09839 & 0.0 & 0.0 & 0.0 & -3.28634 & 0.0 & 1.0
\end{array}\right).
\end{gather*}

The following product of matrices corresponds to the braid $b_{\zeta_5\zeta_6}$:
$$
f'_{n+3}(b_{\zeta_5\zeta_6}) = A_{b_{kl}} = A_{4625}A_{5716}A_{1645}A_{2567}A_{6715}A_{4526}A_{1546}A_{2657}.
$$

In matrix form, this product has the form:

\begin{gather*}
A_{b_{kl}} = \left(\begin{array}{rrrrrrrrr}
1.0 & 0.0 & 0.0 & 0.0 & 0.0 & 0.0 & 0.0 & 0.0 & 0.0 \\
0.0 & 1.0 & 0.0 & 0.0 & 0.0 & 0.0 & 0.0 & 0.0 & 0.0 \\
0.0 & 0.0 & 1.0 & 0.0 & 0.0 & 0.0 & 0.0 & 0.0 & 0.0 \\
0.0 & 0.0 & 0.0 & 1.0 & 0.0 & 0.0 & 0.0 & 0.0 & 0.0 \\
0.0 & 0.0 & 0.0 & 0.0 & 1.0 & 0.0 & 0.0 & 0.0 & 0.0 \\
0.0 & 0.0 & 0.0 & 0.0 & 0.0 & 1.0 & 0.0 & 0.0 & 0.0 \\
0.0 & 0.0 & 0.0 & 0.0 & 0.0 & 0.0 & 0.81649658 & 0.0 & -0.57735027 \\
0.0 & 0.0 & 0.0 & 0.0 & 0.0 & 0.0 & 0.57735027 & 0.0 & 0.81649658 \\
0.0 & 0.0 & 0.0 & 0.0 & 0.0 & 0.0 & 0.0 & 1.0 & 0.0
\end{array}\right) \\
\cdot \left(\begin{array}{rrrrrrrrr}
1.0 & 0.0 & 0.0 & 0.0 & 0.0 & 0.0 & 0.0 & 0.0 & 0.0 \\
0.0 & 1.0 & 0.0 & 0.0 & 0.0 & 0.0 & 0.0 & 0.0 & 0.0 \\
0.0 & 0.0 & 1.0 & 0.0 & 0.0 & 0.0 & 0.0 & 0.0 & 0.0 \\
0.0 & 0.0 & 0.0 & 0.77459667 & 0.0 & 0.0 & 0.0 & 0.0 & -0.63245553 \\
0.0 & 0.0 & 0.0 & 0.63245553 & 0.0 & 0.0 & 0.0 & 0.0 & 0.77459667 \\
0.0 & 0.0 & 0.0 & 0.0 & 1.0 & 0.0 & 0.0 & 0.0 & 0.0 \\
0.0 & 0.0 & 0.0 & 0.0 & 0.0 & 1.0 & 0.0 & 0.0 & 0.0 \\
0.0 & 0.0 & 0.0 & 0.0 & 0.0 & 0.0 & 1.0 & 0.0 & 0.0 \\
0.0 & 0.0 & 0.0 & 0.0 & 0.0 & 0.0 & 0.0 & 1.0 & 0.0
\end{array}\right) \\
\cdot \left(\begin{array}{rrrrrrrrr}
1.0 & 0.0 & 0.0 & 0.0 & 0.0 & 0.0 & 0.0 & 0.0 & 0.0 \\
0.0 & 1.0 & 0.0 & 0.0 & 0.0 & 0.0 & 0.0 & 0.0 & 0.0 \\
0.0 & 0.0 & 1.26491106 & -0.77459667i & 0.0 & 0.0 & 0.0 & 0.0 & 0.0 \\
0.0 & 0.0 & 0.0 & 0.0 & 1.0 & 0.0 & 0.0 & 0.0 & 0.0 \\
0.0 & 0.0 & 0.0 & 0.0 & 0.0 & 1.0 & 0.0 & 0.0 & 0.0 \\
0.0 & 0.0 & 0.0 & 0.0 & 0.0 & 0.0 & 1.0 & 0.0 & 0.0 \\
0.0 & 0.0 & 0.0 & 0.0 & 0.0 & 0.0 & 0.0 & 1.0 & 0.0 \\
0.0 & 0.0 & 0.77459667i & 1.26491106 & 0.0 & 0.0 & 0.0 & 0.0 & 0.0 \\
0.0 & 0.0 & 0.0 & 0.0 & 0.0 & 0.0 & 0.0 & 0.0 & 1.0
\end{array}\right) \\
\cdot \left(\begin{array}{rrrrrrrrr}
1.0 & 0.0 & 0.0 & 0.0 & 0.0 & 0.0 & 0.0 & 0.0 & 0.0 \\
0.0 & 1.0 & 0.0 & 0.0 & 0.0 & 0.0 & 0.0 & 0.0 & 0.0 \\
0.0 & 0.0 & 1.0 & 0.0 & 0.0 & 0.0 & 0.0 & 0.0 & 0.0 \\
0.0 & 0.0 & 0.0 & 1.0 & 0.0 & 0.0 & 0.0 & 0.0 & 0.0 \\
0.0 & 0.0 & 0.0 & 0.0 & 1.0 & 0.0 & 0.0 & 0.0 & 0.0 \\
0.0 & 0.0 & 0.0 & 0.0 & 0.0 & 1.0 & 0.0 & 0.0 & 0.0 \\
0.0 & 0.0 & 0.0 & 0.0 & 0.0 & 0.0 & 1.0 & 0.0 & 0.0 \\
0.0 & 0.0 & 0.0 & 0.0 & 0.0 & 0.0 & 0.0 & 1.29099445i & 1.63299316 \\
0.0 & 0.0 & 0.0 & 0.0 & 0.0 & 0.0 & 0.0 & -1.63299316 & 1.29099445i
\end{array}\right)
\end{gather*}

\begin{gather*}
\cdot \left(\begin{array}{rrrrrrrrr}
1.0 & 0.0 & 0.0 & 0.0 & 0.0 & 0.0 & 0.0 & 0.0 & 0.0 \\
0.0 & 1.0 & 0.0 & 0.0 & 0.0 & 0.0 & 0.0 & 0.0 & 0.0 \\
0.0 & 0.0 & 1.0 & 0.0 & 0.0 & 0.0 & 0.0 & 0.0 & 0.0 \\
0.0 & 0.0 & 0.0 & -1.22474487i & 0.0 & 0.0 & 0.0 & 0.0 & -1.58113883 \\
0.0 & 0.0 & 0.0 & 1.58113883 & 0.0 & 0.0 & 0.0 & 0.0 & -1.22474487i \\
0.0 & 0.0 & 0.0 & 0.0 & 1.0 & 0.0 & 0.0 & 0.0 & 0.0 \\
0.0 & 0.0 & 0.0 & 0.0 & 0.0 & 1.0 & 0.0 & 0.0 & 0.0 \\
0.0 & 0.0 & 0.0 & 0.0 & 0.0 & 0.0 & 1.0 & 0.0 & 0.0 \\
0.0 & 0.0 & 0.0 & 0.0 & 0.0 & 0.0 & 0.0 & 1.0 & 0.0
\end{array}\right) \\
\cdot \left(\begin{array}{rrrrrrrrr}
1.0 & 0.0 & 0.0 & 0.0 & 0.0 & 0.0 & 0.0 & 0.0 & 0.0 \\
0.0 & 1.0 & 0.0 & 0.0 & 0.0 & 0.0 & 0.0 & 0.0 & 0.0 \\
0.0 & 0.0 & 1.0 & 0.0 & 0.0 & 0.0 & 0.0 & 0.0 & 0.0 \\
0.0 & 0.0 & 0.0 & 1.0 & 0.0 & 0.0 & 0.0 & 0.0 & 0.0 \\
0.0 & 0.0 & 0.0 & 0.0 & 1.0 & 0.0 & 0.0 & 0.0 & 0.0 \\
0.0 & 0.0 & 0.0 & 0.0 & 0.0 & 1.22474487 & 0.0 & -0.70710678i & 0.0 \\
0.0 & 0.0 & 0.0 & 0.0 & 0.0 & 0.70710678i & 0.0 & 1.22474487 & 0.0 \\
0.0 & 0.0 & 0.0 & 0.0 & 0.0 & 0.0 & 1.0 & 0.0 & 0.0 \\
0.0 & 0.0 & 0.0 & 0.0 & 0.0 & 0.0 & 0.0 & 0.0 & 1.0
\end{array}\right) \\
\cdot \left(\begin{array}{rrrrrrrrr}
1.0 & 0.0 & 0.0 & 0.0 & 0.0 & 0.0 & 0.0 & 0.0 & 0.0 \\
0.0 & 1.0 & 0.0 & 0.0 & 0.0 & 0.0 & 0.0 & 0.0 & 0.0 \\
0.0 & 0.0 & 0.79056942 & 0.61237244 & 0.0 & 0.0 & 0.0 & 0.0 & 0.0 \\
0.0 & 0.0 & 0.0 & 0.0 & 1.0 & 0.0 & 0.0 & 0.0 & 0.0 \\
0.0 & 0.0 & 0.0 & 0.0 & 0.0 & 1.0 & 0.0 & 0.0 & 0.0 \\
0.0 & 0.0 & 0.0 & 0.0 & 0.0 & 0.0 & 1.0 & 0.0 & 0.0 \\
0.0 & 0.0 & 0.0 & 0.0 & 0.0 & 0.0 & 0.0 & 1.0 & 0.0 \\
0.0 & 0.0 & -0.61237244 & 0.79056942 & 0.0 & 0.0 & 0.0 & 0.0 & 0.0 \\
0.0 & 0.0 & 0.0 & 0.0 & 0.0 & 0.0 & 0.0 & 0.0 & 1.0
\end{array}\right) \\
\cdot \left(\begin{array}{rrrrrrrrr}
1.0 & 0.0 & 0.0 & 0.0 & 0.0 & 0.0 & 0.0 & 0.0 & 0.0 \\
0.0 & 1.0 & 0.0 & 0.0 & 0.0 & 0.0 & 0.0 & 0.0 & 0.0 \\
0.0 & 0.0 & 1.0 & 0.0 & 0.0 & 0.0 & 0.0 & 0.0 & 0.0 \\
0.0 & 0.0 & 0.0 & 1.0 & 0.0 & 0.0 & 0.0 & 0.0 & 0.0 \\
0.0 & 0.0 & 0.0 & 0.0 & 1.0 & 0.0 & 0.0 & 0.0 & 0.0 \\
0.0 & 0.0 & 0.0 & 0.0 & 0.0 & 1.0 & 0.0 & 0.0 & 0.0 \\
0.0 & 0.0 & 0.0 & 0.0 & 0.0 & 0.0 & 1.0 & 0.0 & 0.0 \\
0.0 & 0.0 & 0.0 & 0.0 & 0.0 & 0.0 & 0.0 & 0.79056942 & 0.61237244 \\
0.0 & 0.0 & 0.0 & 0.0 & 0.0 & 0.0 & 0.0 & -0.61237244 & 0.79056942
\end{array}\right)
\end{gather*}

\begin{gather*}
= \left(\begin{array}{rrrrrrrrr}
1.0 & 0.0 & 0.0 & 0.0 & 0.0 & 0.0 & 0.0 & 0.0 & 0.0 \\
0.0 & 1.0 & 0.0 & 0.0 & 0.0 & 0.0 & 0.0 & 0.0 & 0.0 \\
0.0 & 0.0 & 1.0 & 0.7746 & -0.94868 & 0.0 & 0.0 & -0.75 & 0.96825 \\
0.0 & 0.0 & -0.7746 & 1.0 & 1.22474 & 0.0 & 0.7303 & -0.06455 & -1.25 \\
0.0 & 0.0 & 0.94868 & -1.22474 & 1.0 & 0.0 & -0.89443 & 1.26491 & 0.0 \\
0.0 & 0.0 & 0.0 & 0.0 & 0.0 & 1.0 & 0.0 & 0.0 & 0.0 \\
0.0 & 0.0 & 0.0 & -0.7303 & 0.89443 & 0.0 & 1.0 & -0.70711 & 0.91287 \\
0.0 & 0.0 & 0.75 & 0.06455 & -1.26491 & 0.0 & 0.70711 & 1.0 & -1.29099 \\
0.0 & 0.0 & -0.96825 & 1.25 & 0.0 & 0.0 & -0.91287 & 1.29099 & 1.0
\end{array}\right).
\end{gather*}

One can also verify that this relation holds:

$$
A_{b_{ij}}A_{b_{kl}} = A_{b_{kl}}A_{b_{ij}}
$$
\resizebox{1.2\hsize}{!}{$= \left(\begin{array}{rrrrrrrrr}
1.0 & -0.35355 & 0.0 & 0.0 & 1.11803 & 0.27386 & 0.0 & 0.0 & -1.09545 \\
0.35355 & 1.0 & -1.0 & 0.0 & 0.0 & 0.0 & -1.06066 & 0.0 & 0.0 \\
0.0 & 1.0 & 1.0 & 0.7746 & -4.11096 & -0.7746 & 0.0 & -0.75 & 4.06663 \\
0.0 & 0.0 & -0.7746 & 1.0 & 1.22474 & 0.0 & 0.7303 & -0.06455 & -1.25 \\
-1.11803 & 0.0 & 4.11096 & -1.22474 & 1.0 & 0.0 & -4.24853 & 1.26491 & 0.0 \\
-0.27386 & 0.0 & 0.7746 & 0.0 & 0.0 & 1.0 & -0.82158 & 0.0 & 0.0 \\
0.0 & 1.06066 & 0.0 & -0.7303 & 4.24853 & 0.82158 & 1.0 & -0.70711 & 4.19921 \\
0.0 & 0.0 & 0.75 & 0.06455 & -1.26491 & 0.0 & 0.70711 & 1.0 & -1.29099 \\
1.09545 & 0.0 & -4.06663 & 1.25 & 0.0 & 0.0 & -4.19921 & 1.29099 & 1.0
\end{array}\right).$}

\end{appendices}

\pagebreak


\begin{thebibliography} {100}

\bibitem{Aurenhammer-Klein-Lee}
F. Aurenhammer, R. Klein and D.-T. Lee,
{\em Voronoi Diagrams and Delaunay Triangulations},
World Scientific Publishing, 2013, 337 pp.

\bibitem{Fedoseev-Manturov-Nikonov}
D. A. Fedoseev, V. O. Manturov and I. M. Nikonov,
Manifolds of triangulations, braid groups of manifolds, and the groups $\Gamma_n^k$,
preprint, 2020, arXiv:1912.02695v2 [math.GT].

\bibitem{Korepanov-pentagon}
I. G. Korepanov and N. M. Sadykov,
Pentagon relations in direct sums and Grassmann algebras,
{\em SIGMA} \textbf{9} (2013), 030, 16 pp.

\bibitem{Licata-col-braid}
J. Licata and V. V\'ertesi,
Liftable braids and colored braid groupoid,
preprint, 2025, arXiv:2508.05146v1 [math.GT].

\bibitem{Manturov-Non-Reid}
V. O. Manturov,
Non-Reidemeister Knot Theory and Its Applications in Dynamical Systems, Geometry, and Topology,
preprint, 2015, arXiv:1501.05208v1 [math.GT].

\bibitem{Manturov-Knots}
Vassily O. Manturov,
{\em Knot Theory}, Second Edition,
CRC Press, 2018, 580 pp.

\bibitem{Invariants-and-Pictures}
V. Manturov, D. Fedoseev, S. Kim and I. Nikonov,
{\em Invariants and Pictures},
World Scientific Publishing, 2020, 388 pp.

\bibitem{Manturov-Nikonov}
V. O. Manturov and I. M. Nikonov,
The groups $\Gamma_n^4$, braids, and 3-manifolds,
preprint, 2023, arXiv:2305.06316v1 [math.GT].

\bibitem{Rohozhkin}
I. E. Rohozhkin,
Pentagon equations, Delaunay triangulations and pure braid group invariant,
{\em Journal of Knot Theory and Its Ramifications}, 2025, 25 pp.

\bibitem{Tomotada}
T. Ohtsuki,
{\em Quantum Invariants: A Study of Knots, 3-Manifolds, and Their Sets},
World Scientific Publishing, 2002, 508 pp.

\end{thebibliography}
\end{document}